\documentclass[10pt,twoside,reqno]{amsart} 
\usepackage{amsmath, amsfonts, amsthm, amssymb}
\usepackage[english]{babel}
\usepackage[latin1]{inputenc}
\usepackage{amsmath}
\usepackage{amsfonts}
\usepackage{amsthm}
\usepackage[pdftex]{graphicx}
\usepackage{color}
\usepackage{alltt}
\usepackage{caption}
\usepackage{subcaption}

\usepackage{color}
\definecolor{grb}{rgb}   {0.0,   0.99,   0.5 }
\definecolor{mg}{rgb}   {0.85,  0.,    0.85}
\newcommand{\Bk}{\color{black}}

\numberwithin{equation}{section}
\input epsf

 

\newcommand{\norm}[1]{\parallel #1 \parallel}
\newcommand{\ds}{\displaystyle}
\newcommand{\be}{\begin{equation}}
\newcommand{\ee}{\end{equation}}
\newcommand{\ba}{\begin{array}}
\newcommand{\ea}{\end{array}}
\newcommand{\bp}{\begin{proof}}
\newcommand{\ep}{\end{proof}}

\def\boldsymbol{\bf}

\def\f {{\boldsymbol f}}

\def\n{{\boldsymbol n}}

\def\uu{{\boldsymbol u}}
\def\vv{{\boldsymbol v}}
\def\w {{\boldsymbol w}}
\def\x {{\boldsymbol x}}

\def\z{{\boldsymbol z}}

\def\R{{\rm I\hspace{-0.50ex}R} }
\def\P{{\rm I\hspace{-0.50ex}P} }

\def\w{{\bf w}}
\def\f{{\bf f}}
\def\v{{ \bf v}}
\def\u{{\bf u}}
\def\0{{\bf 0}}

\def\z{{\bf z}}

\def\x{{\bf x}}

\def\s{s}
\def\n{{\bf n}}
\def\div{\operatorname{div}}
\def\curl{\operatorname{\bf curl}}

\newcommand{\Om}{{\Omega}}
 
 \newcommand{\dd}{{\rm div}}

\newtheorem{lem}{Lemma}[section]
\newtheorem{thm}[lem]{Theorem}
\newtheorem{hyp}[lem]{Assumption}
\newtheorem{rmq}[lem]{Remark}
\newtheorem{Rem}[lem]{Remark}

\newtheorem{prop}[lem]{Proposition}
\newtheorem{propri}[lem]{Property}

\renewcommand{\epsilon}{\varepsilon}
\def\twoplot[#1]#2#3#4#5{
\begin{figure}[hbt]
\begin{multicols}{2}
\begin{center}
    \includegraphics*[#1]{#2}
    \caption{\label{#2} #4}
\end{center}
\begin{center}
    \includegraphics*[#1]{#3}
    \caption{\label{#3} #5}
\end{center}
\end{multicols}
\end{figure}
}

\begin{document}
\bibliographystyle{plain}
\title[\scshape{A Posteriori error estimates for the coupling system}]
{A Posteriori error estimates for Darcy-Forchheimer's problem coupled with the convection-diffusion-reaction equation }
\author[T. Sayah]{Toni Sayah$^\dagger$}
\author[G. Semaan]{Georges Semaan$^\dagger$}
\author[F. Triki]{Faouzi Triki$^\ddagger$}
\thanks{
\today.
\newline
$^\dagger$ Unit\'e de recherche "Math\'ematiques et Mod\'elisation", CAR, Facult\'e des Sciences,
Universit\'e Saint-Joseph de Beyrouth, B.P
11-514 Riad El Solh, Beyrouth 1107 2050, Liban. \\
E-mails: toni.sayah@usj.edu.lb, georges.semaan2@net.usj.edu.lb,\\
$^\ddagger$ Laboratoire Jean Kuntzmann, UMR CNRS 5224, Universit\'e Grenoble-Alpes, 700 Avenue Centrale, 38401 Saint-Martin-d'H\`eres, France.\\
E-mails: faouzi.triki@univ-grenoble-alpes.fr.\\
}
\maketitle

\begin{abstract}
\noindent In this work we derive {\it a posteriori} error estimates for the convection-diffusion-reaction equation coupled with the Darcy-Forchheimer problem by a nonlinear external source depending on the concentration of the fluid. We introduce the variational formulation associated to the problem, and discretize it by using the finite element method. We prove optimal {\it a posteriori} errors with two types of calculable error indicators. The first one is linked to the linearization and the second one to the discretization. Then we find upper and lower error bounds under additional regularity assumptions on the exact solutions. Finally, numerical computations are performed to show the effectiveness of the obtained error indicators.

\end{abstract}


\section{Introduction.}
\label{intro}
\noindent

This work deals with the {\it a posteriori} error estimate of the Darcy-Forccheimer system coupled with the convection-diffusion-reaction equation. We consider following system of equations:
\begin{equation*}(P)
 \left\{
\begin{array}{ccll}
\medskip
  \ds \frac{\mu}{\rho} K^{-1}  \u + \frac{\beta}{\rho} |\u| \u+ \nabla p & =  & \f(.,C) & \mbox{in}\,\Omega,\\
\medskip
  \dd\, \uu & = & 0 & \mbox{in}\,\Omega,\\
\medskip
  -\alpha \Delta C + \uu\cdot \nabla\, C  + r_0 C& = & g  & \mbox{in}\,\Omega,\\
\medskip
 \uu\cdot\n & = & 0 & \mbox{on}\,\Gamma,\\
 C & = & 0 & \mbox{on}\,\Gamma,
\end{array}
\right.
\end{equation*}
where $\Omega\subset\R^d$, $d=2,3$, is a bounded simply-connected open domain, having a Lipschitz-continuous boundary $\Gamma$ with an outer unit normal $\n$. The unknowns are the velocity $\uu$, the pressure $p$ and
the concentration $C$ of the fluid.
$|.|$ denotes the Euclidean norm, $|\u|^2 = \u \cdot \u$. The parameters $\rho$, $\mu$ and $\beta$ represent the density of the fluid, its viscosity and its dynamic viscosity, respectively. $\beta$ is also referred as Forchheimer number when it is a scalar positive constant. The diffusion coefficient $\alpha$ and the parameter $r_0$ are strictly positive constants. The function $\f$ represents an external force that depends on the concentration~$C$ and the function $g$ represents an external concentration source. $K$ is the permeability tensor, assumed to be uniformly positive definite and bounded such that there exist two positive real numbers $K_m$ and $K_M$ such that
\begin{equation}\label{KmM}
0 < K_m  \le \left\|K^{-1}\right\|_{L^\infty(\Om)^{d\times d}}  \le K_M.\end{equation}
{
It is important to  note that $K_m$ should be smaller than the smallest eigenvalue of $K^{-1}$ over $\Omega$ and $K_M$ could be very large.}\\
System $(P)$  represents the coupling of the Darcy-Forchheimer problem with the convection-diffusion-reaction equation satisfied by the concentration of the fluid. The same system  can represent the coupling of the Darcy-Forchheimer system with the heat equation by replacing the concentration $C$ by the temperature $T$ and setting $r_0=0$.\\
Darcy's law (see \cite{Neum} and \cite{Whit} for the theoretical derivation) is an equation that describes the flow of a fluid through a porous medium. This law was formulated by  Darcy based on experimental results. It is simply the first equation of the system $(P)$ where the dynamic viscosity $\beta=0$. In the case where the velocity of the fluid is higher and the porosity is non uniform, Forchheimer proposed the Darcy-Forchheimer equation (see \cite{Forchheimer}) which is the first equation of system $(P)$ by adding the non-linear term). Several numerical and theoretical studies of the Darcy-Forchheimer equation were performed, and among others we mention \cite{GW,LMJ,HH,SJ,sayahdarcyF}. \\
For the coupling of Darcy's equation with the heat equation, we refer to \cite{SC} where the system is treated using a spectral method.
The authors in \cite{BDGHMS} and \cite{DD} considered the same stationary system but coupled with a nonlinear viscosity that depends on the temperature. In \cite{sayahdib}, the authors derived an optimal {\it a posteriori} error estimate for each of the numerical schemes proposed in \cite{BDGHMS}. We can also refer to \cite{3O} where the authors used a vertex-centred finite volume method to discretize the coupled system.
For physical applications of system $(P)$, we can refer to \cite{physapp}. In \cite{semaan}, we introduced the variational formulation associated to system $(P)$, and we showed uniqueness under additional constraints on the concentration. Then, we discretized the system by using the finite element method and we showed the existence and uniqueness of corresponding solutions. Moreover, we established the {\it{a priori}} error estimate between the exact and numerical solutions and introduced a numerical scheme where we studied the corresponding convergence.  \\
 I. Babu$\check{s}$ka was the first  who introduced \emph{a posteriori} analysis (see \cite{Babu78C}), then it was developed by {R.~Verf\"{u}rth} \cite{Verfurth1996}, and has been the object of a large number of publications. Many works have established the \emph{a posteriori} error estimates for the Darcy flow, see for instance \cite{alonso, braess, carstensen, lovadina}. In \cite{DGHS}, the authors established {\it a posteriori} error estimates for Darcy's problem coupled with the heat equation. Sayah T. (see \cite{SAY}) established the \emph{a posteriori} error estimates for the Brinkman-Darcy-Forchheimer problem.  Moreover, in \cite{semaan2}, we established the {\it{a posteriori}} estimates for the Darcy-Forchheimer problem without the convection-diffusion-reaction equation. Furthermore, several works established the {\it a priori} and {\it a posteriori} errors for the time-dependent convection-diffusion-reaction equation coupled with Darcy's equation(see \cite{7, 77}).\\\\
 The main goal of this work is to derive the {\it{a posteriori}} error estimates associated to the coupling system $(P)$ for the numerical scheme introduced in \cite{semaan}. We start by recalling some auxiliary results from \cite{semaan} concerning the discretization of system $(P) $, and the numerical scheme with the corresponding convergence. In a second step, we establish the \emph{a posteriori} error estimates where the error between the exact and iterative numerical solutions are bounded by two types of local indicators: the indicators of discretization and the indicators of linearization. Then, we show the corresponding efficiency by bounding each indicator by the local error.
Finally, we present some numerical computations in order to show the effectiveness of the proposed method. \\

\noindent The outline of the paper is as follows:
\begin{itemize}
\item Section \ref{sec:prob cont} is devoted to the continuous problem.
\item In section \ref{sec:discretisation}, we introduce the discrete and iterative problems and recall their
main properties.
\item  In section \ref{sec:posteriori}, we provide the error indicators and prove the upper and lower error bounds.
\item  Numerical results validating the theory are presented in Section \ref{numerique}.
\end{itemize}
We further assume that the volumic and boundary sources verify the following conditions:
\begin{hyp}\label{hypdata}
The functions $\f$ and $g$ verify:
\begin{enumerate}
\item $\f$ can be written as follows:
\begin{equation}\label{formf}
\forall \x \in \Omega, \forall C\in \R,\qquad\f(\x,C)  = \f_0(\x) + \f_1(C),
\end{equation}
 where $\f_0 \in L^{\frac{3}{2}}(\Omega)^d$ and ${\f}_1$ is Lipschitz-continuous with constant $c_{\f_1}>0$, and verifies ${\f}_1(0)=0$.  In particular we have
\[
\forall \xi \in \R, \, |\f_1(\xi)| \le c_{\f_1} |\xi |,
\]

\item $g\in L^2(\Omega)$.
\end{enumerate}
\end{hyp}

\section{Variational Formulation}
\label{sec:prob cont}
In order to introduce the variational formulation, we recall some classical Sobolev spaces and their properties.
Let $\x=(x_1,x_2, \cdots, x_d)\in \R^d $, and let  $\alpha=(\alpha_1,\alpha_2, \dots ,\alpha_d)$ be a $d$-uplet of non negative
integers, set $|\alpha|=\sum_{i=1}^d \alpha_i$, and define the partial
derivative $\partial^\alpha$ by
$$
\partial^{\alpha}=\ds \frac{\partial^{|\alpha|}}{\partial x_1^{\alpha_1}\partial
x_2^{\alpha_2}\dots\partial x_d^{\alpha_d}}.
$$
Then, for any positive integer $m$ and number $p\geq 1$, we recall
the classical Sobolev space \cite{s1,s1'}
\begin{equation}
\label{eq:Wm,p}
W^{m,p}(\Omega)=\{v \in L^p(\Omega);\,\forall\,|\alpha|\leq
m,\;\partial^{\alpha} v \in L^p(\Omega)\},
\end{equation}
equipped with the seminorm
\begin{equation}
\label{eq:semnormWm,p}
|v|_{W^{m,p}(\Omega)}=\big(\sum_{|\alpha|=m} \int_{\Omega}
|\partial^{\alpha} v|^p \,d\x\,\big)^{\frac{1}{p}}
\end{equation}
and the norm
\begin{equation}
\label{eq:normWm,p}
\|v\|_{W^{m,p}(\Omega)}=\big(\sum_{0\leq k\leq m}
|v|_{W^{k,p}(\Omega)}^p \big)^{\frac{1}{p}}.
\end{equation}
When $p=2$, this space is the Hilbert space $H^m(\Omega)$.
The definitions of these spaces are extended straightforwardly to
vectors, with the same notation, but with the following
modification for the norms in the non-Hilbert case. Let $\vv$ be a
vector valued function; we set
\begin{equation}
\label{eq:normLp}  \|\vv\|_{L^p(\Omega)^d} \Bk = \big(\int _{\Omega}
|\vv|^p\,d\x\, \big)^{\frac{1}{p}},
\end{equation}
where $|.|$ denotes the Euclidean vector norm.
\\
For vanishing boundary values, we define
\begin{equation}
\label{eq:H10}
\begin{array}{ccl}
H^1_0(\Omega) & = & \{v\in H^1(\Omega);\, v_{|_{\Gamma}}=0\},\\
W^{1,q}_0(\Omega) & = & \{v\in W^{1,q}(\Omega);\, v_{|_{\Gamma}}=0\}.\\
\end{array}
\end{equation}
We shall often use the following Sobolev imbeddings: for any real number
$p\geq 1$ when $d=2$, or $1 \le p \le \frac{2\,d}{d-2}$ when $d\ge 3$, there exist constants $S_p$ and $S_p^0$ such that
\begin{equation}
\label{eq:Sp}
\forall\, v\in H^1(\Omega),\,\,\|v\|_{L^p(\Omega)}\leq
{S_p} \|v\|_{H^1(\Omega)}
\end{equation}
and
\begin{equation}
\label{pcc}
\forall\,v\in H^1_0(\Omega),\,\,\|v\|_{L^p(\Omega)}\leq
{S_p^0} |v|_{H^1(\Omega)}.
\end{equation}
When $p=2$, (\ref{pcc}) reduces to Poincar\'e's inequality.\\
To deal with the Darcy-Forchheimer, we recall the space
\begin{equation}
\label{eq:L20}
L^2_0(\Omega)  =  \big\{v\in L^2(\Omega);\,\ds \int _{\Omega}
v\,d\x\,=0\big\}.
%
\end{equation}
%
%
%
%
It follows from the nonlinear
term in the system $(P)$ that the velocity $\u$ and the test function $\v$ must belong to $L^3(\Omega)^d$; then, the gradient of the pressure must belong to $L^{\frac{3}{2}}(\Omega)^d$. Furthermore, the concentration $C$ must be in $H^1_0(\Omega)$. Thus, we introduce the spaces (see \cite{GW})
\[
X=L^3(\Omega)^d,\quad
M=W^{1,\frac{3}{2}}(\Omega) \cap L^2_0(\Omega),\quad
Y=H^1_0(\Omega).
\]
Furthermore, we recall the following { inf-sup} condition between $X$ and $M$ (see \cite{GW}),
\begin{equation} \label{infsup}
\ds \inf_{q \in M} \,  \sup_{\v \in X} \ds
\frac{\ds \int_{\Omega}\v(\x)\cdot\nabla q(\x)\,d\x} {\left\|\v\right\|_{L^3(\Omega)^d}\left\|\nabla q\right\|_{L^\frac{3}{2}(\Omega)}}=1.
\end{equation}
We introduce the following variational formulation associated to problem $(P)$:
\begin{equation*}(V_a)
\begin{cases}
\medskip
\text{Find} \,(\u,p,C) \in X \times M \times Y \text{ such that:} \\
\medskip
\begin{aligned}
\forall \v \in X, \; \frac{\mu}{\rho}\int_{\Omega}(K^{-1}\u(\x))\cdot \v(\x)\,d\x + \frac{\beta}{\rho}\int_{\Omega}|\u(\x)|
\u(\x)\cdot\v(\x)\,d\x &\ds +\int_{\Omega}\nabla p(\x)\cdot\v(\x)\,d\x\\
&=\int_{\Omega}\f (\x,C(\x))\cdot\v(\x)\,d\x,
\end{aligned}\\
\medskip
\forall q \in M, \; \ds \int_{\Omega}\nabla q(\x)\cdot\u(\x)\,d\x=0, \\
\begin{aligned}
\forall S \in Y, \; \alpha \int_{\Omega} \nabla C(\x)\cdot\nabla S(\x)\,d\x + \int_{\Omega}(\u \cdot \nabla C)(\x)S(\x)\,d\x
+ & r_0\int_{\Omega}C(\x)S(\x)\,d\x 
=\int_{\Omega}g(\x)S(\x)\,d\x.
\end{aligned}
\end{cases}
\end{equation*}
The existence and uniqueness of the solutions to the problem $(V_a)$ can be found in \cite{semaan}.
To study the discretization of the  variational problem $(V_a)$, it is convenient to introduce the nonlinear mapping:
\begin{equation*}
    \begin{array}{rccl}
  \mathcal{A}:&  L^3(\Omega)^d &\mapsto&L^\frac{3}{2}(\Omega)^d \\
    &\v &\mapsto &\mathcal{A}(\v)=\ds \frac{\mu}{\rho}K^{-1}\v+\frac{\beta}{\rho}|\v|\v.
    \end{array}
\end{equation*}
We refer to \cite{GW,monot} for the following useful results.
\begin{propri}
$\mathcal{A}$ satisfies the following properties:
\begin{enumerate}
\item $\mathcal{A}$ maps $L^3(\Omega)^d$ into $L^{\frac{3}{2}}(\Omega)^d$ and we have for all $\v \in L^3(\Omega)^d$:
$$
\left\|\mathcal{A}(\v)\right\|_{L^{\frac{3}{2}}(\Omega)^d} \leq
\frac{\mu}{\rho}\left\|K^{-1}\right\|_{\infty}
\left\|\v\right\|_{L^{\frac{3}{2}}(\Omega)^d}+
\frac{\beta}{\rho}\left\|\v\right\|_{L^{3}(\Omega)^d}^2.
$$
\item For all $(\v,\w) \in \mathbb{R}^d \times \mathbb{R}^d$, we have,
 \begin{equation} \label{deuxestimA}
|\mathcal{A}(\v)-\mathcal{A}(\w)| \leq
\left( \frac{\mu}{\rho}\left\|K^{-1}\right\|_{\infty} + \frac{2\beta}{\rho}(|\v|+|\w|)\right)|\v-\w|.
\end{equation}
\item $\mathcal{A}$ is monotone from $L^{3}(\Omega)^d$ into $L^{\frac{3}{2}}(\Omega)^d$, and we have for all $\v,\w \in L^{3}(\Omega)^d$,
$$
\int_{\Omega} (\mathcal{A}(\v(\x))-\mathcal{A}(\w(\x)))\cdot(\v(\x)-\w(\x))\,d\x \geq
\max(c_m \left\|\v-\w\right\|_{L^3(\Omega)^d}^3, \ds \frac{\mu}{\rho}K_m\left\|\v-\w\right\|_{L^2(\Omega)^d}^2),
$$
where $c_m$ is a strictly positive constant.
\item  $\mathcal{A}$ is coercive in $L^{3}(\Omega)^d$:
{
\[
\ds \underset{\left\|\u\right\|_{L^3(\Omega)^d \to \infty}}{\lim} \frac{ \ds \int_{\Omega}\mathcal{A}(\u)\cdot \u \,d\x}{\left\|\u\right\|_{L^3(\Omega)^d}}= \ds +\infty.
\]
}

\item  $\mathcal{A}$ is  hemi-continuous in $L^{3}(\Omega)^d$: for fixed $\u,\v \in L^{3}(\Omega)^d$, the mapping
\[
t \longrightarrow \ds \int_\Omega \mathcal{A}(\u + t \v)\cdot \v \,d\x
\]
is continuous from $\R$ into $\R$.
\end{enumerate}
\end{propri}

\section{Discretization} \label{sec:discretisation}

In this section, we recall the discretization of problem $(P)$ introduced in \cite{semaan}, and restrict the analysis to dimensions $d=2, 3$. We further assume that $\Omega$ is a polygon when $d=2$ or polyhedron when $d=3$, so it can be completely meshed. Next, we describe the space discretization.  A regular family of triangulations (see Ciarlet~\cite{PGC})
$( \mathcal{T}_h )_h$ of $\Omega$, is a set of closed non degenerate
triangles for $d=2$ or tetrahedra for $d=3$, called elements,
satisfying
\begin{itemize}
\item for each $h$, $\bar{\Omega}$ is the union of all elements of
 $\mathcal{T}_h$;
\item the intersection of two distinct elements of  $\mathcal{T}_h$ is
either empty, a common vertex, or an entire common edge (or face
when $d=3$);
\item the ratio of the diameter   {  $h_{\kappa}$  of an element $\kappa \in \mathcal{T}_h$  to
the diameter  $\rho_\kappa$}  of its inscribed circle when $d=2$ or ball when $d=3$
is bounded by a constant independent of $h$, that is, there exists a strictly positive constant $\sigma$ independent of $h$ such that,
\begin{equation}
\label{eq:reg}
{
\ds \max_{\kappa \in \mathcal{T}_{h}} \frac{h_{ \kappa}}{\rho_{ \kappa}} \le \sigma.}
\end{equation}
\end{itemize}
As usual, $h$ denotes the maximal diameter
of all elements of $\mathcal{T}_{h}$.   To define the finite element functions,  let $r$ be a non negative integer.  For each {  $\kappa$ in $\mathcal{T}_{h}$, we denote by  $\mathbb{P}_r(\kappa)$  the space of restrictions to $\kappa$ of polynomials in $d$ variables and total degree at most  $r$, with a similar notation on the faces or edges of $\kappa$}. For every edge (when $d=2$) or face (when $d=3$) $e$ of the mesh  $\mathcal{T}_h$, we denote by $h_e$ the diameter of $e$.\\
\noindent In order to use inverse inequalities, we assume that the family or triangulations is uniformly regular in the following sens: there exists $\beta_0 >0$  such that, for every element $\kappa\in {\mathcal T}_h$, we have
\begin{equation}  \label{meshunifregular} h_\kappa \ge \beta_0 h.
\end{equation}

\noindent  We shall use the following inverse inequality:  for any numbers $p,q \geq 2$,  for any dimension $d$, and for any non negative integer $r$,  there exist  constants $c_I(p)>0$, $c_J(q)>0$, and $c_L>0$ such that for any polynomial function $v_h$ of degree $r$ on an element $\kappa$ or an edge (when $d = 2$) or face (when $d = 3$) $e$ of the mesh $\mathcal{T}_{h}$,
\begin{equation}
\label{eq:inversin}
\begin{array}{ll}
\medskip
 \|v_h\|_{L^p(\kappa)}\leq  c_I(p)  h_\kappa^{\frac{d}{p}-\frac{d}{2}}\|v_h\|_{L^2(\kappa)},\\
 \|v_h\|_{L^q(e)}\leq  c_J(q)  h_e^{\frac{d-1}{q}-\frac{d-1}{2}}\|v_h\|_{L^2(e)}, \\
 |v_h|_{H^1(\kappa)} \leq c_L h_{\kappa}^{\frac{d}{2}-\frac{d}{p}-1}\left\|v_h\right\|_{L^p(\kappa)},
 \end{array}
\end{equation}
where $c_I$, $c_J$ and $c_L$ depend on the regularity parameter $\sigma$ of \eqref{eq:reg}.\\

Let $X_{h} \subset X$, $M_{h}\subset M$ and $Y_h \subset Y$ be the discrete spaces corresponding to the velocity, the pressure and the concentration.
\subsection{Discrete Scheme}
We recall the discrete problem introduced in \cite{semaan}:
 Find $(\u_h,p_h,C_h)\in X_h \times M_h \times Y_h$ such that
\begin{equation}
(V_{ah})
\left\{
\begin{array}{ll}
\medskip
\forall \v_h \in X_{h}, \quad \ds \int_\Omega \mathcal{A} (\u_h) \cdot \v_h d\x
+ \int_\Omega \nabla p_h \cdot \v_h \, d\x = \ds \int_\Omega \f(C_h) \cdot \v_h \, d\x,\\
\medskip
\forall q_h\in M_{h}, \quad \ds \int_\Omega \nabla q_h \cdot \u_h \, d\x = 0,\\
\ds \forall S_h \in Y_h, \; \alpha \int_{\Omega} \nabla C_h\cdot\nabla S_h\,d\x + \int_{\Omega}(\u_h\cdot \nabla C_h)S_h\,d\x
+  \frac{1}{2} \int_\Omega \div(\u_h) \, C_h S_h d\x\\
\hspace{8cm} \ds  + r_0\int_{\Omega}C_h S_h\,d\x = \ds \int_{\Omega}g S_h\,d\x.
\end{array}
\right.
\end{equation}

In the following, we will introduce the finite dimension spaces $X_h,M_h$ and $Y_h$. Let $\kappa$ be an element of ${\mathcal T}_h$ with vertices $a_i$, $1\leq i \leq d+1$, and
corresponding barycentric coordinates $\lambda_i$. We denote by $b_\kappa \in \mathbb{P}_{d+1}(\kappa)$ the basic bubble function :
\begin{equation}
\label{eqn:bbl}
b_\kappa(\x)=\lambda_1(\x)...\lambda_{d+1}(\x).
\end{equation}
We observe that $b_\kappa(\x)=0$ on $\partial \kappa$ and that $b_\kappa(\x)>0$
in the interior of $\kappa$.\\
We introduce the following discrete spaces:
\begin{equation}\label{defespaceprem}
\begin{split}
X_{h}=&\{\vv_h \in
(\mathcal{C}^0(\bar{\Omega}))^d;\;\forall\, \kappa \in \mathcal{T}_h,\;
\vv_{h}|_\kappa\in {\mathcal{P}(\kappa)}^d\},\\
M_{h}=&\{  q_h\in C^0(\bar{\Omega}); \, \forall\, \kappa \in \mathcal{T}_h, \;q_{h}|_\kappa \in \P_1 (\kappa)\} \cap L^2_0(\Omega), \\
Y_h =& \{q_h\in C^0(\bar{\Omega}); \, \forall\, \kappa \in \mathcal{T}_h, \;q_{h}|_\kappa \in \P_1 (\kappa)\}\cap H^1_0(\Omega),\\
V_{h} =& \{ \v_h \in X_{h}; \forall q_h \in M_{h}, \ds \int_\Omega \nabla q_h \cdot \v_h \, d\x= 0  \},
\end{split}
\end{equation}
where
$$\mathcal{P}(\kappa)= \mathbb{P}_1(\kappa) \oplus {\rm Vect}\{b_\kappa\}.$$
In this case,  the following inf-sup condition holds \cite{GR}:
\begin{equation}
\label{infsuph1}
\forall\, q_h\in M_{h},\; \sup_{\v_h\in X_{h}}\ds \frac{\ds \int_{\Omega} \nabla q_h \cdot \v_h \,  d\x\,}{\|\v_h\|_{X_h}}\geq \beta_2 \|q_h\|_{M_{h}},
\end{equation}
where $\beta_2$ is a strictly positive constant independent of $h$.\\
The existence and uniqueness of solutions of problem $(V_{ah})$  can be deduced
from inf-sup condition \eqref{infsuph1}  (see for instance  \cite{semaan}).\\

We shall use the following results (see \cite{semaan} and \cite{semaan2}):
\begin{enumerate}
\item For the concentration: there exists an approximation operator (when $d=2$, see Bernardi
and Girault~\cite{s5} or Cl\'ement~\cite{s3}; when $d =2$ or $d=3$, see Scott and Zhang~\cite{ZC}), $R_h$ in ${\mathcal L}(W^{1,p}(\Omega); Y_{h})$ such that for all $\kappa$ in $\mathcal{T}_h$, $m=0,1$, $l=0,1$, and all $p\ge 1$,
\begin{equation}\label{1.2} \forall\, S \in
W^{l+1,p}(\Omega),\,\,|S-R_h(S)|_{W^{m,p}(\kappa)}\leq c(p,m,l)
\,h^{l+1-m}|S|_{W^{l+1,p}(\Delta_\kappa)},
\end{equation}
where $\Delta_\kappa$ is the macro element containing the values of $S$ used in defining $R_h(S)$. Furthermore for all $\kappa$ in  $\mathcal{T}_h$, for all $e$ in $\partial \kappa$ and for all $ S \in H^{1}(\Omega)$,
\begin{equation}\label{estgam}
\| S - R_h S \|_{L^{2}(e)} \le c_e h_e^{1/2} |S |_{H^1(w_e)},
\end{equation}
where $c_e>0$ is aconstant independent of $h$.
\item For the velocity: We introduce a variant of $R_h$ denoted by $\mathcal{F}_h$ {that is stable over $L^p(\Omega)^d$ for all $p \geq 1$ (see Apprendix in \cite{GiLi})} and satisfying \eqref{1.2}.
%
%
\item For the pressure: Let $r_h$ be a Cl\'ement-type interpolation operator \cite{s3}. 
We have the following error estimate: for all $\kappa$ in  $\mathcal{T}_h$, for all $e$ in $\partial \kappa$ and for all $ q \in W^{1,3/2}(\Omega)$,
\begin{equation}\label{estom}
\| q - r_h q \|_{L^{3/2}(\kappa)} \le c_\kappa h_\kappa |q |_{W^{1,3/2}(w_\kappa)}
\end{equation}
and
\begin{equation}\label{estgam1}
\| q - r_h q \|_{L^{3/2}(e)} \le c_e h_e^{1/3} |q |_{W^{1,3/2}(w_e)},
\end{equation}
where $c_e>0$ and $c_\kappa>0$ are constants independent of $h$.
\end{enumerate}
%
%
%
%
%

 We recall the following theorem of {\it a priori} error estimates \cite{semaan}:
\begin{thm}\label{thmprioriconv} Under Assumption \ref{hypdata}, let $(\u_h,p_h,C_h)$  be a solution of problem $(V_{ah})$, and  $(\u,p,C)$  be a solution of problem $(V_{ah})$. If $(\u,p,C)$ are such that $ C \in H^2(\Om)$, $\u \in W^{1,3}(\Omega)^d$ and $p\in H^2(\Om)$, and satisfies the following condition:
\begin{equation}\label{priorie}
\ds S_6^0 |C|_{W^{1,3}(\Omega)} + \norm{C}_{L^\infty(\Omega)} \le \ds \frac{\alpha \mu K_m}{2\sqrt{2} \rho c_{\f_1} S_2^0},
\end{equation}
then, we have the following {\it a priori} error estimates:
\begin{equation}\label{estpriori1}
|C-C_h|_{H^1(\Omega)}  + \norm{\u - \u_h}_{L^2(\Omega)^d} + \norm{\nabla(p-p_h)}_{L^{\frac{3}{2}}(\Omega)^d} \le c_1 h
\end{equation}
and
\begin{equation}\label{estpriori2}
 \norm{\u - \u_h}_{L^3(\Omega)^d}  \le c_2 h^{2/3},
\end{equation}
where $c_1$ and $c_2$ are strictly positive constants independent of $h$.
\end{thm}
The following proposition gives a bound for the discrete velocity in $L^6(\Omega)^d$ which will be used in Section \ref{sec:posteriori}.
\begin{prop} Under the assumptions of Theorem \ref{thmprioriconv}, we have the following bound:
\begin{equation} \label{uhl6}
\left\|\u_h\right\|_{L^6(\Omega)^d} \leq \hat{c}(\u,p,C)
\end{equation}
where $\hat{c}(\u,p,C)$ is a positive constant independent of $h$.
\end{prop}
\begin{proof}
As $\u \in W^{1,3}(\Omega)^d \subset L^6(\Omega)^d$, and using a triangle inequality, we start from the following bound
\begin{equation*}
    \left\|\u_h\right\|_{L^6(\Omega)^d}  \leq \left\|\u_h-\mathcal{F}_hu\right\|_{L^6(\Omega)^d}+ \left\|\mathcal{F}_hu\right\|_{L^6(\Omega)^d}.
\end{equation*}
Using the fact that the operator $\mathcal{F}_h$ is stable over $L^6(\Omega)^d$, and that the mesh is uniformly regular, we get the bound
\begin{equation}
    \begin{aligned}
     \left\|\u_h\right\|_{L^6(\Omega)^d} & \leq
     ch^{-d/3}\left\|\u_h-\mathcal{F}_hu\right\|_{L^2(\Omega)^d}+ c\left\|\u\right\|_{L^6(\Omega)^d} \\
     &\leq
     ch^{-d/3}(\left\|\u_h-\u\right\|_{L^2(\Omega)^d}+\left\|\u-\mathcal{F}_hu\right\|_{L^2(\Omega)^d})+c\left\|\u\right\|_{L^6(\Omega)^d}.
    \end{aligned}
\end{equation}
Finally, using the \emph{a priori} error estimates, the properties of the operator $\mathcal{F}_h$, we get the desired result.

$\hfill\Box$

\end{proof}
\subsection{Successive approximations}
As the problem is nonlinear, we introduce a straightforward successive approximation algorithm (see \cite{semaan}) which converges to the discrete solution under suitable conditions. The algorithm proceeds as follows: let $\u^h_0\in X_h$ and $C^0_h\in Y_0$ the initial guesses. Having  $(\u_h^{i}, C_h^{i}) \in X_h \times Y_h$ at each iteration $i$, we compute $(\u_h^{i+1}, p_h^{i+1}, C_h^{i+1}) \in X_h \times M_h \times Y_h$, such that
\begin{equation}
(V_{ahi})
\left\{
\begin{array}{ll}
\medskip
\forall \v_h \in X_{h}, \quad \ds \gamma \int_\Om (\u_h^{i+1}-\u_h^i)\cdot \v_h d\x +  \frac{\mu}{\rho}\int_{\Omega}(K^{-1}\u_h^{i+1})\cdot \v_h\,d\x + \frac{\beta}{\rho}\int_{\Omega}|\u_h^i| \u_h^{i+1}\cdot\v_h\,d\x \\
\medskip
\hspace{8cm} \ds  + \int_\Omega \nabla p_h^{i+1} \cdot \v_h \, d\x = \ds \int_\Omega \f(C_h^i) \cdot \v_h \, d\x,\\
\medskip
\forall q_h\in M_{h}, \quad \ds \int_\Omega \nabla q_h \cdot \u_h^{i+1} \, d\x = 0,\\
\ds \forall S_h \in Y_h, \; \alpha \int_{\Omega} \nabla C_h^{i+1}\cdot\nabla S_h\,d\x + \int_{\Omega}(\u_h^{i+1}\cdot \nabla C_h^{i+1})S_h\,d\x
+  \frac{1}{2} \int_\Omega \div(\u_h^{i+1}) \, C_h^{i+1} S_h d\x\\
\hspace{8cm} \ds  + r_0\int_{\Omega}C_h^{i+1} S_h\,d\x = \ds \int_{\Omega}g S_h\,d\x,
\end{array}
\right.
\end{equation}
where $\gamma$ is a real strictly positive parameter. Later on, the parameter $\gamma$ will be chosen to ensure the convergence of algorithm $(V_{ahi})$. At each iteration $i$, having $\u_h^i$ and $C_h^i$, the first two lines of $(V_{ahi})$ computes $(\u_h^{i+1},p_h^{i+1})$. Next, we substitute $\u_h^{i+1}$ by its value in the third equation of $(V_{ahi})$ to compute $C_h^{i+1}$. \\
For the existence and uniqueness of solutions of problem $(V_{ahi})$, we recall the following two Theorems \cite{semaan} with few modifications that give the parameter $\gamma$ explicitly. The main ideas in the proof are exactly the same but for the reader convenience, we detailed some steps.
\begin{thm}\label{thm41}
In addition to assumption \ref{hypdata}, we suppose that $\f_0 \in L^2(\Omega)^d$. For each $(\u_h^i,C_h^i)\in X_h \times Y_h$, problem $(V_{ahi})$ admits a unique solution $(\u_h^{i+1}, p_h^{i+1}, C_h^{i+1}) \in X_h \times M_h \times Y_h$. Moreover, we have the following bound
\begin{equation}\label{boundC}
|C_h^{i+1}|_{1,\Omega} \leq \frac{S_2^0}{\alpha} \left\|g\right\|_{L^2(\Omega)}.
\end{equation}
Furthermore, if the initial value $\u_h^0$ satisfies the condition
\begin{equation} \label{uhi0} \left\|\u_h^0\right\|_{L^2(\Omega)^d} \leq L_1(\f,g),\end{equation}
where
\[ L_1(\f,g)=\frac{\rho}{\mu K_m}(\left\|\f_0\right\|_{L^2(\Omega)^d} + c_{\f_1}\frac{(S_2^0)^2}{\alpha}\left\|g\right\|_{L^2(\Omega)}), \]
and if $\gamma\geq \gamma_{*}$ with
\begin{equation} \begin{aligned}\label{condh1}
\gamma_{*}=\frac{32 \beta}{27 \rho}&c_I^3h^{-d/2} \left(\frac{\rho}{\mu K_m}+\frac{\rho K_M}{\mu K_m^2}\right)
\left(\left\|\f_0\right\|_{L^2(\Omega)^d} + c_{\f_1}\frac{(S_2^0)^2}{\alpha}\left\|g\right\|_{L^2(\Omega)}\right)\\
&+ \frac{32 \beta^2}{27 \rho \mu^3 K_m^3}c_I^{6}h^{-d}
\left(\left\|\f_0\right\|_{L^2(\Omega)^d} + c_{\f_1}\frac{(S_2^0)^2}{\alpha}\left\|g\right\|_{L^2(\Omega)}\right)^2,
\end{aligned}
\end{equation}
then, the following inequalities hold
\begin{equation} \label{estimuhi2} \left\|\u_h^{i+1}\right\|_{L^2(\Omega)^d} \leq L_1(\f,g), \end{equation}
and
\begin{equation} \label{estimuhi3}
\left\|\u_h^{i+1}\right\|_{L^3(\Omega)^d}^3\leq \frac{\mu K_m}{\beta}L_1^2(\f,g).
\end{equation}
\end{thm}
\begin{proof}
For the existence, uniqueness of solutions to problem $(V_{ahi})$, and the bound \eqref{boundC},  we refer to \cite{semaan}.
Now, we focus on the remaining part of the proof,
using Theorem 4.1 in \cite{semaan}, we recall the following inequality
\begin{equation}
\begin{array}{ll}
\ds \gamma \left\|\u_h^{i+1}-\u_h^i\right\|_{L^2(\Omega)^d}^2 + \frac{\mu K_m}{\rho}\left\|\u_h^{i+1}-\u_h^i\right\|_{L^2(\Omega)^d}^2  \leq \left\|\f(C_h^i)\right\|_{L^2(\Omega)^d} \left\|\u_h^{i+1}-\u_h^i\right\|_{L^2(\Omega)^d} \\
\hspace{2cm} \ds + \frac{\mu K_M}{\rho} \left\|\u_h^{i+1}-\u_h^i\right\|_{L^2(\Omega)^d}\left\|\u_h^i\right\|_{L^2(\Omega)^d} + \frac{\beta}{\rho}C_I^3h^{-d/2}\left\|\u_h^i\right\|_{L^2(\Omega)^d}^2\left\|\u_h^{i+1}-\u_h^i\right\|_{L^2(\Omega)^d}. \end{array}
\end{equation}
We simplify by $\left\|\u_h^{i+1}-\u_h^i\right\|_{L^2(\Omega)^d}$ to obtain:
\[
\begin{array}{ll}
\medskip
\ds (\gamma+\frac{\mu K_m}{\rho}) \| \u_{h}^{i+1} -\u_{h}^i\|_{L^2(\Omega)^d}   \le
\ds \| \f(C_h^i) \|_{L^2(\Omega)^d}
+  \frac{\mu K_M}{\rho}  \|\u_{h}^i \|_{L^2(\Omega)^d}  +
\frac{\beta}{\rho} {  c_I^3 }  h^{-\frac{d}{2}}   \|\u_{h}^i\|^2_{L^2(\Omega)^d}.
\end{array}
\]
Using the properties of $\f$ and the bound of the concentration \eqref{boundC}, we get the following estimate:
\begin{equation}\label{bound11}
\ds  \| \u_{h}^{i+1} -\u_{h}^i\|_{L^2(\Omega)^d} \le L_2(\f,g, \| \u_h^i \|_{L^2(\Omega)^d}),
\end{equation}
where
$$ L_2(\f,g, \eta) = \ds \frac{\rho }{\mu K_m}  \Big(\| \f_0 \|_{L^2(\Omega)^d}+c_{\f_1}\frac{(S_2^0)^2}{\alpha}\left\|g\right\|_{L^2(\Omega)}
+  \frac{\mu K_M}{\rho}  \eta +
\frac{\beta}{\rho} {  c_I^3 }  h^{-\frac{d}{2}}\eta^2  \Big), \;\; \eta\in \mathbb R_+. $$
Then, we are now in position to show relation \eqref{estimuhi2}. We consider the first equation of problem $(V_{ahi})$ with $\v_h=\u_h^{i+1}$, and  obtain
\begin{equation} \begin{aligned}
\gamma \int_{\Omega}(\u_h^{i+1}-\u_h^i)\cdot\u_h^{i+1}\,d\x + \frac{\mu}{\rho}\int_{\Omega}K^{-1}\u_h^{i+1}\cdot\u_h^{i+1}\,d\x + \frac{\beta}{\rho}\left\|\u_h^{i+1}\right\|_{L^3(\Omega)^d}^3 &= \int_{\Omega}\f(C_h^i)\cdot\u_{h}^{i+1}\,d\x \\
& + \frac{\beta}{\rho}\int_{\Omega}(|\u_h^{i+1}|-|\u_h^i|) |\u_h^{i+1}|^2\,d\x.
\end{aligned}
\end{equation}
Using the properties of $K^{-1}$, the Cauchy-Shwartz inequality and the relations $\ds ab \leq \frac{1}{2\varepsilon}a^2 + \frac{\varepsilon}{2}b^2$ and $\ds a^2b \leq \frac{1}{3}(\frac{1}{\delta^3}b^3+2\delta^{\frac{3}{2}}a^3)$ wih $ \ds \varepsilon = \frac{\mu K_m}{\rho} $ and $\ds \delta = (\frac{3\beta}{4\rho})^{2/3}$, we get
\begin{equation} \begin{aligned}
& \frac{\gamma}{2}\left\|\u_h^{i+1}\right\|_{L^2(\Omega)^d}^2 - \frac{\gamma}{2}\left\|\u_h^i\right\|_{L^2(\Omega)^d}^2 + \frac{\gamma}{2}\left\|\u_h^{i+1}-\u_h^i\right\|_{L^2(\Omega)^d}^2 + \frac{\mu K_m}{2\rho} \left\|\u_h^{i+1}\right\|_{L^2(\Omega)^d}^2 + \frac{\beta}{2\rho} \left\|\u_{h}^{i+1}\right\|_{L^3(\Omega)^d}^3 \\ &\hspace{4cm} \leq \frac{\rho}{2\mu K_m}\left\|\f(C_h^i)\right\|_{L^2(\Omega)^d}^2 + \frac{16 \beta}{27\rho} c_I^3h^{-d/2}\left\|\u_{h}^{i+1}-\u_{h}^i\right\|_{L^2(\Omega)^d}^3.
\end{aligned} \end{equation}
We denote by $$ \ds c_1(\left\|\u_h^i\right\|_{L^2(\Omega)^d}) = \frac{\gamma}{2}-\frac{16\beta}{27\rho}c_I^3 h^{-d/2}L_2(\f,g,\left\|\u_h^i\right\|_{L^2(\Omega)^d})$$
which is not necessarily positive at this level.
Therefore, by using the bound \eqref{bound11}, we obtain the following bound:
\begin{equation}\begin{aligned} \label{rel412}
 \frac{\gamma}{2}\left\|\u_h^{i+1}\right\|_{L^2(\Omega)^d}^2 &- \frac{\gamma}{2}\left\|\u_h^i\right\|_{L^2(\Omega)^d}^2 + c_1(\left\|\u_h^i\right\|_{L^2(\Omega)^d})\left\|\u_h^{i+1}-\u_h^i\right\|_{L^2(\Omega)^d}^2 + \frac{\mu K_m}{2\rho} \left\|\u_h^{i+1}\right\|_{L^2(\Omega)^d}^2 \\ & + \frac{\beta}{2\rho} \left\|\u_{h}^{i+1}\right\|_{L^3(\Omega)^d}^3  \leq  \frac{\mu K_m}{2 \rho}L_1^2(\f,g).
\end{aligned} \end{equation}
We now prove estimate \eqref{estimuhi2} by induction on $i \geq 1$ under some condition on $\gamma$ that we will determine. Starting with relation \eqref{uhi0}, we suppose that we have
\begin{equation} \label{induct} \left\|\u_h^i\right\|_{L^2(\Omega)^d} \leq L_1(\f,g). \end{equation}
We have two situations:
\begin{itemize}
\item $ \ds \left\|\u_h^{i+1}\right\|_{L^2(\Omega)^d} \leq \left\|\u_h^i\right\|_{L^2(\Omega)^d}$, which immediately leads to
\[\left\|\u_h^{i+1}\right\|_{L^2(\Omega)^d} \leq L_1(\f,g).\]
\item $\ds \left\|\u_h^{i+1}\right\|_{L^2(\Omega)^d} \geq \left\|\u_h^i\right\|_{L^2(\Omega)^d}  $. By using the induction condition \eqref{induct} and the fact that the function $L_2$ is increasing with respect to $\eta$, we chose
\begin{equation} \begin{aligned} \label{conditiongam}
\frac{\gamma}{2} & \geq \frac{16\beta}{27\rho}c_I^3h^{-d/2}L_2(\f,g, L_1(\f,g)) \\
&\geq\frac{16\beta}{27\rho}c_I^3h^{-d/2} L_2(\f,g,\left\|\u_h^i\right\|_{L^2(\Omega)^d}),
\end{aligned} \end{equation}
to get $ \ds c_1(\left\|\u_h^i\right\|_{L^2(\Omega)^d})\geq0$, and deduce from relation \eqref{rel412} that
\[ \left\|\u_h^{i+1}\right\|_{L^2(\Omega)^d} \leq L_1(\f,g).\]
\end{itemize}
Then relation \eqref{estimuhi2} holds. The bound \eqref{estimuhi3} is a simple consequence of \eqref{rel412} and \eqref{estimuhi2}.\\
Now we focus on the inequality \eqref{conditiongam}. It is easy to show that $\forall \eta \in \mathbb{R}_{+}$,
\begin{equation}
 \ds L_2(\f,g, \eta) \leq
 \ds \frac{ \rho}{\mu K_m}( \| \f_0 \|_{L^2(\Omega)^d} +\frac{c_{\f_1}(S_2^0)^2}{\alpha}\left\|g\right\|_{L^2(\Omega)})+  \frac{K_M}{K_m} \eta + \frac{\beta}{\rho} c_I^3 h^{-\frac{d}{2}} \frac{\rho}{\mu K_m} \eta^2,
\end{equation}
and then to get by using the definition of $L_1$:
\begin{equation} \begin{aligned} L_2(\f,g,L_1(\f,g)) &\leq (\frac{\rho}{\mu K_m}+\frac{\rho K_M}{\mu K_m^2})\left( \left\|\f_0\right\|_{L^2(\Omega)^d}+c_{\f_1}\frac{(S_2^0)^2}{\alpha}\left\|g\right\|_{L^2(\Omega)}
\right)\\
&+
\frac{\beta\rho^2}{\mu^3 K_m^3}c_I^3h^{-d/2}
\left( \left\|\f_0\right\|_{L^2(\Omega)^d}+c_{\f_1}\frac{(S_2^0)^2}{\alpha}\left\|g\right\|_{L^2(\Omega)}
\right)^2.
\end{aligned} \end{equation}
Relation \eqref{conditiongam} and the last inequality allow us to obtain
\[
\frac{\gamma}{2} - \frac{16\beta}{27\rho}c_I^3h^{-d/2}L_2(\f,g, L_1(\f,g)) \geq \phi(\gamma)
\]
with
\begin{equation} \begin{aligned}
\phi(\gamma)= \frac{\gamma}{2}&-\frac{16 \beta}{27 \rho}c_I^3h^{-d/2} \left(\frac{\rho}{\mu K_m}+\frac{\rho K_M}{\mu K_m^2}\right)
\left(\left\|\f_0\right\|_{L^2(\Omega)^d} + c_{\f_1}\frac{(S_2^0)^2}{\alpha}\left\|g\right\|_{L^2(\Omega)}\right)\\
&- \frac{16 \beta^2}{27 \rho \mu^3 K_m^3}c_I^{6}h^{-d}
\left(\left\|\f_0\right\|_{L^2(\Omega)^d} + c_{\f_1}\frac{(S_2^0)^2}{\alpha}\left\|g\right\|_{L^2(\Omega)}\right)^2.
\end{aligned}
\end{equation}
We remark that $\phi(\gamma)$ is a polynomial of first degree with respect to $\gamma$ with only  root $\gamma_{*}>0$.
Finally, we get $\phi(\gamma)\geq0$ for all $\gamma \geq \gamma_*$.

$\hfill\Box$

\end{proof}

The next theorem shows the convergence of the solution $(\u_h^{i},p_h^i,C_h^i)$
of problem $(V_{ahi})$ to the solution of problem $(V_{ah})$.
\begin{thm} \label{thm42} Under the assumption of Theorem \ref{thm41}, we assume that the concentration solution of the problem $(V_a)$ satisfies
\begin{equation} \label{condsurci}
S_6^0|C|_{W^{1,3}(\Omega)}+\left\|C\right\|_{L^{\infty}(\Omega)}
\leq \frac{\mu K_m \alpha }{2\rho c_{\f_1}S_2^0}.
\end{equation}
 Moreover, if $\gamma$ satisfies the condition
\begin{equation} \label{condsurgamai}
\gamma > \max\{\gamma_{*},\gamma_{**}\},
\end{equation}
where $$
\gamma_{**}=2c_I^4h^{-2d/3}\frac{\rho^{1/3}\beta^{4/3}}{\mu^{5/3}K_m^{5/3}}\left(\left\|\f_0\right\|_{L^2(\Omega)^d}+c_{\f_1}\frac{(S_2^0)^2}{\alpha}\left\|g\right\|_{L^2(\Omega)}
\right)^{4/3}$$ and if
\begin{equation} \label{condsurhi}
h \leq \big(\frac{1}{2c_Ic_1}(|C|_{W^{1,3}(\Omega)}+\frac{\left\|C\right\|_{L^{\infty}(\Omega)}}{S_6^0})\big)^{6/(6-d)},
\end{equation}
where $c_1$ is the constant in \eqref{estpriori1}, then the solution $(\u_h^i,p_h^i, C_h^i)$ of problem $(V_{ahi})$ converges in $L^2(\Omega)^d\times L^2(\Omega) \times H^1(\Omega)$ to the solution of problem $(V_{ah})$.
\begin{proof}
Again, using Theorem 4.2 in \cite{semaan}, we start from the following bound:
\begin{equation} \label{utilepourconvi} |C_h-C_h^{i+1}|_{1,\Omega} \leq
\frac{S_6^0}{2\alpha}\left[ 2c_Ih^{-d/6}|C-C_h|_{1,\Omega} + |C|_{W^{1,3}(\Omega)} +\frac{ \left\|C\right\|_{L^{\infty}(\Omega)}}{S_6^0}\right]\left\|\u_h^{i+1}-\u_h\right\|_{L^2(\Omega)^d}.\end{equation}
Furthermore, by taking the difference between the first equations of problems $(V_{ah})$ and $(V_{ahi})$ with $\v_h=\u_h^{i+1}-\u_h$, we get
\begin{equation} \begin{aligned}
&\frac{\gamma}{2}\left\|\u_h^{i+1}-\u_h\right\|_{L^2(\Omega)^d}^2 - \frac{\gamma}{2}\left\|\u_h^i-\u_h\right\|_{L^2(\Omega)^d}^2 + \frac{\gamma}{2}\left\|\u_h^{i+1}-\u_h\right\|_{L^2(\Omega)^d}^2 + \frac{\mu}{\rho}\int_{\Omega}K^{-1}|\u_h^{i+1}-\u_h|^2\,d\x \\
&
\ds + \frac{\beta}{\rho} \int_\Om (|\u_h^{i}|-|\u_h^{i+1}|)\u_h^{i+1}\cdot(\u_h^{i+1}-\u_h))d\x +  \frac{\beta}{\rho} \int_\Om (|\u_h^{i+1}|\u_h^{i+1}-|\u_h|\u_h)\cdot(\u_h^{i+1}-\u_h)d\x\\
& \ds \hspace{8cm} = \int_{\Omega}(\f(C_h^i)-\f(C_h)).(\u_h^{i+1}-\u_h)\,d\x.
\end{aligned} \end{equation}
By using the monotonicity property of the operator $\mathcal{A}$ we obtain,
\begin{equation} \begin{aligned}
&\frac{\gamma}{2}\left\|\u_h^{i+1}-\u_h\right\|_{L^2(\Omega)^d}^2 - \frac{\gamma}{2}\left\|\u_h^i-\u_h\right\|_{L^2(\Omega)^d}^2 + \frac{\gamma}{2}\left\|\u_h^{i+1}-\u_h^i\right\|_{L^2(\Omega)^d}^2 + \frac{\mu K_m}{\rho} \left\|\u_h^{i+1}-\u_h\right\|_{L^2(\Omega)^d}^2\\
&\leq \frac{\beta}{\rho}\left\|\u_h^{i+1}-\u_h^i\right\|_{L^3(\Omega)^d}\left\|\u_h^{i+1}\right\|_{L^3(\Omega)^d}\left\|\u_h^{i+1}-\u_h\right\|_{L^3(\Omega)^d} + c_{\f_1}S_2^0|C_h^i-C_h|_{1,\Omega}\left\|\u_h^{i+1}-\u_h\right\|_{L^2(\Omega)^d} \\
&  \leq \frac{\beta}{\rho}c_I^2h^{-d/3}(\frac{\mu K_m}{\beta})^{1/3}(L_1(\f,g))^{2/3}\left\|\u_h^{i+1}-\u_h^i\right\|_{L^2(\Omega)^d}\left\|\u_h^{i+1}-\u_h\right\|_{L^2(\Omega)^d}\\
&+ c_{\f_1}S_2^0|C_h^i-C_h|_{1,\Omega}\left\|\u_h^{i+1}-\u_h\right\|_{L^2(\Omega)^d}.
\end{aligned} \end{equation}
We denote by $\ds c_2=\frac{\beta}{\rho}(\frac{\mu K_m}{\beta})^{1/3}c_I^2 L_1^{2/3}(\f,g)$ and we use the relation $\ds ab \leq \frac{1}{2\varepsilon}a^2+\frac{\varepsilon}{2}b^2$ with $\ds \varepsilon=\frac{\mu K_m}{2\rho}$, we get
\begin{equation} \begin{aligned}
&\frac{\gamma}{2}\left\|\u_h^{i+1}-\u_h\right\|_{L^2(\Omega)^d}^2 - \frac{\gamma}{2}\left\|\u_h^i-\u_h\right\|_{L^2(\Omega)^d}^2 + \frac{\gamma}{2}\left\|\u_h^{i+1}-\u_h^i\right\|_{L^2(\Omega)^d}^2 + \frac{\mu K_m}{2\rho} \left\|\u_h^{i+1}-\u_h\right\|_{L^2(\Omega)^d}^2\\
& \hspace{2cm} \leq \frac{\rho c_2^2}{\mu K_m}h^{-2d/3}\left\|\u_h^{i+1}-\u_h^i\right\|_{L^2(\Omega)^d}^2 + \frac{\rho (c_{\f_1}S_2^0)^2}{\mu K_m} |C_h^i-C_h|_{1,\Omega}^2.
\end{aligned} \end{equation}
We then choose
\begin{equation} \label{choixdegam}
\ds \frac{\gamma}{2}>\frac{\rho c_2^2}{\mu K_m}h^{-2d/3},
\end{equation}
and  denote by $\ds c_3= \frac{\gamma}{2}-\frac{\rho c_2^2}{\mu K_m}h^{-2d/3} >0$, to conclude that
\begin{equation} \begin{aligned} \label{426}
\frac{\gamma}{2}\left\|\u_h^{i+1}-\u_h\right\|_{L^2(\Omega)^d}^2 - \frac{\gamma}{2}\left\|\u_h^i-\u_h\right\|_{L^2(\Omega)^d}^2 + c_3\left\|\u_h^{i+1}-\u_h^i\right\|_{L^2(\Omega)^d}^2 &+  \frac{\mu K_m}{2\rho} \left\|\u_h^{i+1}-\u_h\right\|_{L^2(\Omega)^d}^2\\
& \leq \frac{\rho (c_{\f_1} S_2^0)^2}{\mu K_m} |C_h^i-C_h|_{1,\Omega}^2.
\end{aligned} \end{equation}
Combining \eqref{426} with \eqref{utilepourconvi} and using the{ \it a priori} error estimate \eqref{estpriori1}, we get
\begin{equation}\label{condint}
\begin{aligned}
\frac{\gamma}{2}\left\|\u_h^{i+1}-\u_h\right\|_{L^2(\Omega)^d}^2 &- \frac{\gamma}{2}\left\|\u_h^i-\u_h\right\|_{L^2(\Omega)^d}^2 + c_3\left\|\u_h^{i+1}-\u_h^i\right\|_{L^2(\Omega)^d}^2 +  \frac{\mu K_m}{2\rho} \left\|\u_h^{i+1}-\u_h\right\|_{L^2(\Omega)^d}^2\\
& \leq \frac{\rho}{\mu K_m} (\frac{c_{\f_1}S_2^0S_6^0}{2\alpha})^2 \left[ 2c_Ic_1h^{(6-d)/6}+|C|_{W^{1,3}(\Omega)}+\frac{\left\|C\right\|_{L^{\infty}(\Omega)}}{S_6^0}\right]^2 \left\|\u_h^i-\u_h\right\|_{L^2(\Omega)^d}^2.
\end{aligned} \end{equation}
Thus, Assumptions \eqref{condsurci} and \eqref{condsurhi} allow us to get
\begin{equation}\label{lastine}
\begin{array}{l}
\medskip
\ds (\frac{\gamma}{2}+\frac{\mu K_m}{4\rho})\big(\left\|\u_h^{i+1}-\u_h\right\|_{L^2(\Omega)^d}^2-
\left\|\u_h^i-\u_h\right\|_{L^2(\Omega)^d}^2\big) + c_3\left\|\u_h^{i+1}-\u_h^i\right\|_{L^2(\Omega)^d}^2  \\
\hspace{8cm} \ds + \frac{\mu K_m}{4\rho}  \left\|\u_h^{i+1}-\u_h\right\|_{L^2(\Omega)^d}^2\leq 0.
\end{array}
\end{equation}
Finally,  \eqref{choixdegam} is clearly satisfied if $\gamma>\gamma_{**}$. The remaining part of the proof is detailed in \cite{semaan}.

$\hfill\Box$

\end{proof}
\end{thm}
\section{A posteriori error estimation}
\label{sec:posteriori}
The  \emph{a posteriori} analysis controls the overall discretization error of a problem by providing error indicators  that are  easy to compute.  Once these error indicators are constructed, their efficiency can be proven by bounding each indicator by the local error.
\noindent As usual, for {\it a posteriori } error estimates, we introduce the following notations. We denote by
\begin{itemize}
\item $\Gamma_h^i$ the set of edges (when $d=2$) or faces (when
$d=3$) of $\kappa$ that are not contained in $\partial \Omega$.
\item $\Gamma_h^b$  the set of edges (when $d=2$) or faces (when
$d=3$) of $\kappa$ which are contained in $\partial \Omega$.
\end{itemize}
For every element $\kappa$ in  $\mathcal{T}_h$, we denote by $w_\kappa$ the union of elements $K$ of $\mathcal{T}_h$ such that $\kappa \cap K \ne \phi$. Furthermore, for every edge (when $d=2$) or face (when $d=3$) $e$ of the mesh
 $\mathcal{T}_h$, we denote by
\begin{itemize}
\item $\omega_e$ the union of elements of $ \mathcal{T}_h$  adjacent to
$e$.
\item $[\cdot]_e$ the jump through $e\in \Gamma_h^i$. 
\end{itemize}

In this and the next  sections, the {\it a posteriori} error estimates are established  for  slightly smoother solutions.
%
\subsection{Upper error bound}
\noindent In order to establish upper bounds, we introduce, on every edge  ($d=2$) or face ($d=3$) $e$ of the mesh, the function
\begin{equation}
\label{eq:phih}
 \phi_{h,1}^e=
 \left\{
\begin{array}{ccl}
\medskip
\ds \frac{1}{2}\,[\textbf{u}_h^{i+1}\cdot\textbf{n}]_{e} \quad \mbox{if } e \in \Gamma_h^i,\\
 \textbf{u}_h^{i+1}\cdot\textbf{n} \quad \mbox{if } e  \in \Gamma_h^b.
\end{array}
\right. \end{equation}

A standard calculation shows that the solutions of problems $(V_a)$ and $(V_{ahi})$ satisfy for all $(\v,q,S) \in X\times M\times Y$ and $(\v_h,q_h,S_h) \in X_h \times M_h\times Y_h $:
\begin{equation} \label{residuconcent}
\begin{aligned}
&\alpha \int_{\Omega}\nabla (C-C_h^{i+1})\cdot \nabla S\,d\x + \int_{\Omega}(\u\cdot \nabla C)S\,d\x- \int_{\Omega}(\u_h^{i+1}\cdot \nabla C_h^{i+1})S\,d\x + r_0\int_{\Omega}CS\,d\x- r_0\int_{\Omega}C_h^{i+1}S\,d\x  \\&
- \frac{1}{2}\int_{\Omega} \div \u_h^{i+1}C_h^{i+1}S\,d\x =
\sum_{\kappa \in \mathcal{T}_h} \Big[
\int_{\kappa}(g-g_h)(S-S_h)\,d\x
-\frac{\alpha}{2} \sum_{e \in \partial \kappa \cap \Gamma_h^i }
\int_{e} [\nabla C_h^{i+1}\cdot\n]_e (S-S_h)\, ds \\ &
+
 \int_{\kappa} (\alpha \triangle C_h^{i+1}-\u_h^{i+1}\cdot \nabla C_h^{i+1}-\frac{1}{2}\div \u_h^{i+1}C_h^{i+1}-r_0C_h^{i+1}+g_h)(S-S_h)d\x\Big],
\end{aligned}
\end{equation}
\\
\begin{equation} \label{residuvitesse}
\begin{aligned}
& \frac{\mu}{\rho} \int_{\Omega} K^{-1}(\u-\u_h^{i+1})\cdot \v \,d\x + \frac{\beta}{\rho} \int_{\Omega} (|\u|\u - |\u_h^{i}|\u_h^{i+1})\cdot \v \,d\x + \int_{\Omega} \nabla (p-p_h^{i+1})\cdot \v \,d\x \\
& = \sum_{\kappa \in \mathcal{T}_h} \Big[
\int_{\kappa} \big( - \nabla p_h^{i+1} -\gamma (\u_h^{i+1}-\u_h^i) - \frac{\mu}{\rho} K^{-1}\u_h^{i+1}- \frac{\beta}{\rho} |\u_h^i|\u_h^{i+1}+\f_h(., C_h^i) \big) \cdot (\v - \v_h)\,d\x \\
&
+ \gamma \int_{\kappa } (\u_h^{i+1}-\u_h^i)\cdot \v \,d\x + \int_{\kappa} \big(\f(.,C)-\f_h(.,C)\big)\cdot \v \,d\x +
\int_{\kappa} \big( \f_h( ., C) - \f_h(. , C_h^i) \big)\cdot \v \,d\x
\\
&
+ \int_{\kappa} \big( \f_h(. , C_h^i) - \f(. , C_h^i) \big) \cdot \v_h \, d\x
\Big],
\end{aligned}
\end{equation}
and
\begin{equation} \label{residupression}
\int_{\Omega} \nabla q \cdot (\u - \u_h^{i+1}) \,d\x = \sum_{\kappa \in \mathcal{T}_h} \Big[ \int_{\kappa} (q-q_h)\div \u_h^{i+1}\,d\x -
\sum_{e \in \partial \kappa} \int_{e} \phi_{h,1}^e (q-q_h)ds \Big],
\end{equation}
where $g_h$ is an approximation of $g$ which is constant on each triangle $\kappa $ of $\mathcal{T}_h$ and $\f_h$ is an approximation of $\f$ given by :
\[
\forall C \in Y, \quad
\f_h(.,C)= \frac{1}{|\kappa|}\int_{\kappa}\f(.,C)\,d\x.
\]
Since $\f_1$ is $c_{\f_1}$ Lipschitz, we clearly have the following property : \begin{equation} \label{fhlipshitz}
\left\|\f_h(C_1)-\f_h(C_2)\right\|_{L^2(\kappa)} \leq c_{\f_1} \left\|C_1-C_2\right\|_{L^2(\kappa)} \quad \forall (C_1, C_2) \in Y\times Y.
\end{equation}
From the error equations we deduce the following error indicators for each $\kappa \in \mathcal{T}_h$:

   \begin{equation} \label{premindic}
   \eta_{\kappa,i}^{(L_1)}= \left\|\u_h^{i+1}-\u_h^i\right\|_{L^2(\kappa)}, \end{equation}

    \begin{equation}\eta_{\kappa,i}^{(L_2)}= ||C_h^i-C_h^{i+1}||_{H^1(\kappa)},
    \end{equation}

    \begin{equation}
        \begin{aligned}
  \eta_{\kappa , i}^{(D_1)}= &h_{\kappa} \left\| \alpha \triangle C_h^{i+1}-\u_h^{i+1}\cdot \nabla C_h^{i+1}-\frac{1}{2}\div \u_h^{i+1}C_h^{i+1}-r_0C_h^{i+1}+g_h \right\|_{L^2(\kappa)}\\
    &+ \frac{1}{2} \sum_{e \in \partial \kappa \cap \Gamma_h^i }h_e^{\frac{1}{2}}\left\|\alpha[\nabla C_h^{i+1}\cdot \n]_e \right\|_{L^2(e)},
    \end{aligned}
    \end{equation}
    \begin{equation}
          \eta_{\kappa , i}^{(D_2)}= \left\|- \nabla p_h^{i+1} -\gamma (\u_h^{i+1}-\u_h^i) - \frac{\mu}{\rho} K^{-1}\u_h^{i+1}- \frac{\beta}{\rho} |\u_h^i|\u_h^{i+1}+\f_h(., C_h^i)
          \right\|_{L^2(\kappa)},
    \end{equation}
    and
    \begin{equation} \label{dernindic}
        \eta_{\kappa , i}^{(D_3)}=h_\kappa \left\|\div \u_h^{i+1}\right\|_{L^3(\kappa)} + \sum_{e \in \partial \kappa} h_e^{\frac{1}{3}}\left\|\phi_{h,1}^e\right\|_{L^3(e)}.
    \end{equation}

    In order to establish the upper bound, we need to first  bound the numerical solution $\u_h^{i+1}$ in $L^6(\Omega)^d$ in terms of the exact solution which is the subject of the next lemma.
    \begin{lem} \label{uborne6}
    Let the mesh satisfy \eqref{eq:reg}, under the assumptions of Theorem \ref{thmprioriconv}, Theorem \ref{thm41}, Theorem \ref{thm42}, there exists an integer $i_0$ depending on $h$ such that for all $i \geq i_0$, the numerical velocity $\u_h^{i+1}$ satisfies the following bound :
    \begin{equation}
    \left\|\u_h^{i+1}\right\|_{L^6(\Omega)^d} \leq \hat{c}_1(\u,p,C)
    \end{equation}
    where $\hat{c}_1$ is a constant depending on the exact solution $(\u,p,C)$ of problem $(V_a)$.
    \end{lem}
    \begin{proof}
      Let $(\u,p,C)$ be the solution of problem $(V_a)$, $(\u_h,p_h,C_h)$ be the solution of $(V_{ah})$ and $(\u_h^{i+1},p_h^{i+1},C_h^{i+1})$ the solution of $(V_{ahi})$. \\
By using \eqref{eq:inversin} (for $p=6$) and the bound \eqref{uhl6},
the term $\|\u^{i+1}_h \|_{L^6(\Omega)}$ can be bounded as following:
\begin{equation} \label{equat412}
\begin{aligned}
    \left\|\u_h^{i+1}\right\|_{L^6(\Omega)^d} &\leq \left\|\u_h^{i+1}-\u_h\right\|_{L^6(\Omega)^d} + \left\|\u_h\right\|_{L^6(\Omega)^d} \\
    & \leq c_I h^{-d/3}\left\|\u_h^{i+1}-\u_h\right\|_{L^2(\Omega)^d} + \hat{c}(\u,p,C).
    \end{aligned}
\end{equation}
As  $\u^{i+1}_h$ converges to $\u_h$ in $L^2(\Omega)^d$,  there exists an integer $i_0$ depending on $h$ such that for all $i\ge i_0$ we have
\begin{equation}\label{equat66}
\| \u^{i+1}_h - \u_h \|_{L^2(\Omega)^d} \le h^{d/3}.
\end{equation}
Then inequality \eqref{equat412} gives by using  inequality \eqref{equat66} the desired result.

$\hfill\Box$

%
\end{proof}
Our main goal is to get an upper bound of the error between the exact solution $(\u,p,C)$ of problem $(V_a)$ and the numerical solution $(\u_h^{i+1},p_h^{i+1},C_h^{i+1})$ of $(V_{ahi})$ in $X \times M \times Y$. To obtain the desired result, we need to first establish an upper bound  for the error between the exact and numerical solutions in $L^2(\Omega)^d \times M \times Y$.
\begin{lem} \label{existvr}
There exists a velocity $\v_r$ in $L^3(\Omega)^d$ that solves the following variational problem:
\begin{equation} \begin{aligned} \label{upbdvit2}
 \forall q \in M, \quad \int_{\Omega} \nabla q\cdot \v_r \,d\x =
 \sum_{\kappa \in \mathcal{T}_h} \Big( \int_{\kappa }(q-r_hq)\div \u_h^{i+1}\,d\x - \sum_{e \in \partial_{\kappa}}\int_{e}\phi_{h,1}^{e}(q-r_{h}q)\,ds \Big)
 \end{aligned}
 \end{equation}
 which satisfies the following bound:
 \begin{equation} \begin{aligned} \label{upbdvit4}
     \norm{\v_r}_{L^3(\Omega)^d} \leq  c'_2 \Big( \sum_{\kappa \in \mathcal{T}_h}  \big( \eta_{\kappa,i}^{(D_3)} \big) \Big).
     \end{aligned}
 \end{equation}
 \begin{proof}
Using equation \eqref{residupression} with $q_h=r_hq$, considering the inf-sup condition \eqref{infsup}, and the fact that the right hand side term is a continuous linear function of $q$, we deduce the existence a velocity $\v_r$ in $X$ such that \eqref{upbdvit2} is verified,  and satisfying the following bound:
 \begin{equation} \begin{aligned} \label{upbdvit3}
     \norm{\v_r}_{L^3(\Omega)^d} \leq \sup_{q \in M} \frac{1}{\|\nabla q\|_{L^{3/2}(\Omega)^3}}\Big|
     \sum_{\kappa \in \mathcal{T}_h}& \big[ \norm{q - r_hq}_{L^{3/2}(\kappa)} \norm{\div \u_h^{i+1}}_{L^3(\kappa)}
     \\
     &+ \sum_{e \in \partial \kappa} \norm{\phi_{h,1}^{e}}_{L^{3}(e)}
     \norm{q - r_hq}_{L^{3/2}(e)} \big] \Big|.
     \end{aligned}
 \end{equation}
 Thus, from the properties of the operator $r_h$, the regularity of $\mathcal{T}_h$, and the following H\"{o}lder inequality,
 $(p=3/2 , q=3)$
 \[
 \sum_{k=1}^n a_kb_k \leq \big(\sum_{k=1}^n a_k^p \big)^{1/p} \big(\sum_{k=1}^n b_k^q \big)^{1/q},
 \]
 we get after cubing the last equation:
 \begin{equation}
 \norm{\v_r}_{L^3(\Omega)^d}^3 \leq c_2 \sum_{\kappa \in \mathcal{T}_h}  \big( \eta_{\kappa,i}^{(D_3)} \big)^3,
 \end{equation}
 with $c_2>0$ is a constant independent of $h$.
 Finally, we obtain  the wanted result by taking the cubic root of the previous inequality.

$\hfill\Box$

 \end{proof}

\end{lem}
\begin{thm} \label{thmupbd1}
Under the assumptions of Lemma \ref{uborne6}, we suppose in addition that the exact solution $(\u,C)$ of Problem $(V_a)$ satisfies: $\u \in L^{\infty}(\Omega)^d$  and
%
\begin{equation} \label{conditionC}
S_6^0|C|_{W^{1,3}(\Omega)}+\norm{C}_{L^{\infty}(\Omega)} \leq \frac{\alpha\mu K_m}{8\rho c_{\f_1}^2S_2^0}.
\end{equation}
Then, there exists an integer $i_0$ depending on $h$ such that $\forall i \geq i_0$, the solutions $(\u,p,C)$
and $(\u_h^{i+1},p_h^{i+1}, C_h^{i+1})$ of problems $(V_a)$ and $(V_{ahi})$ satisfy the following error inequality
\begin{equation} \begin{aligned} \label{upperbound1}
\norm{\u-\u_h^{i+1}}_{L^2(\Omega)^d}+& \left\|C-C_h^{i+1}\right\|_{H^1(\Omega)} + \norm{\v_r}_{L^3(\Omega)^d}\leq
\tilde{c}_3 \Big( \sum_{\kappa \in \mathcal{T}_h} \big( \eta_{\kappa ,i}^{(D_1)}+\eta_{\kappa ,i}^{(D_2)}+\eta_{\kappa ,i}^{(D_3)}
 + \eta_{\kappa ,i}^{(L_1)}\\
 &+\eta_{\kappa ,i}^{(L_2)}
 + h_k \norm{g-g_h}_{L^2(\kappa)}+\norm{\f(.,C)-\f_h(.,C)}_{L^2(\kappa)}\big) \Big)
 \end{aligned}
\end{equation}
where $\tilde{c}_3$ is a positive constant independent of $h$.
\end{thm}
\begin{proof}
Let us start with the concentration equation \eqref{residuconcent} tested with $S=C-C_{h}^{i+1}$ and $S_h=R_h(S)$. It can be written as:
\begin{equation}
    \begin{aligned} \label{420}
&\alpha | C-C_h^{i+1}|^2_{H^1(\Omega)} + r_0 \left\|C-C_h^{i+1}\right\|_{L^2(\Omega)}^2 \\&
\qquad  =
-\int_{\Omega}(\u\cdot \nabla C)S\,d\x + \int_{\Omega}(\u_h^{i+1}\cdot \nabla C_h^{i+1})S\,d\x
+\frac{1}{2}\int_{\Omega} \div \u_h^{i+1}C_h^{i+1}S\,d\x \\
&
\qquad \qquad+\sum_{\kappa \in \mathcal{T}_h} \Big[
\int_{\kappa}(g-g_h)(S-S_h)\,d\x
-\frac{\alpha}{2} \sum_{e \in \partial \kappa \cap \Gamma_h^i }
\int_{e} [\nabla C_h^{i+1}\cdot\n]_e (S-S_h)\, ds \\ &
\qquad \qquad +
 \int_{\kappa} (\alpha \triangle C_h^{i+1}-\u_h^{i+1}\cdot \nabla C_h^{i+1}-\frac{1}{2}\div \u_h^{i+1}C_h^{i+1}-r_0C_h^{i+1}+g_h)(S-S_h)d\x\Big].
\end{aligned}
\end{equation}

The first three terms of the right-hand side in the previous equation can be written as :
\begin{equation} \begin{aligned} \label{upbd1}
& - \int_{\Omega}(\u\cdot \nabla C)S\,d\x+ \int_{\Omega}(\u_h^{i+1}\cdot \nabla C_h^{i+1})S\,d\x
+ \frac{1}{2}\int_{\Omega} \div \u_h^{i+1}C_h^{i+1}S\,d\x\\
&
\qquad = -\int_{\Omega}((\u - \u_h^{i+1})\cdot \nabla C)(C-C_h^{i+1})\,d\x \\
&
\qquad \qquad - \int_{\Omega}(\u_h^{i+1}\cdot \nabla (C-C_h^{i+1})(C-C_h^{i+1})\,d\x
     + \frac{1}{2}\int_{\Omega} \div \u_h^{i+1}C_h^{i+1}(C-C_h^{i+1})\,d\x.
    \end{aligned}
\end{equation}
By applying Green's formula and using the fact that the fluid is incompressible, the last two terms of \eqref{upbd1} can be written as:
\begin{equation}
\begin{aligned}
  -\int_{\Omega}(\u_h^{i+1}\cdot \nabla (C-C_h^{i+1}) (C-C_h^{i+1})\,d\x
 &+ \frac{1}{2}\int_{\Omega} \div \u_h^{i+1}C_h^{i+1}(C-C_h^{i+1})\,d\x\\
 &=\frac{1}{2}\int_{\Omega}\div \u_h^{i+1}C(C-C_h^{i+1})\,d\x
 \\
 & = -\frac{1}{2}\int_{\Omega} \div (\u - \u_h^{i+1})C(C-C_h^{i+1})\,d\x
 \\
 &=\frac{1}{2}\int_{\Omega}(\u - \u_h^{i+1})\cdot\big(
 \nabla C(C-C_h^{i+1})+C\nabla (C-C_h^{i+1})\big)\,d\x.
 \end{aligned}  \end{equation}
 Thus, the right-hand side of Equation \eqref{upbd1} can be written as
 \begin{equation} \begin{aligned} \label{upbd2}
    &-\int_{\Omega}((\u - \u_h^{i+1})\cdot \nabla C)(C-C_h^{i+1})\,d\x -
    \int_{\Omega}(\u_h^{i+1}\cdot \nabla (C-C_h^{i+1})(C-C_h^{i+1})\,d\x \\
    &
     + \frac{1}{2}\int_{\Omega} \div \u_h^{i+1}C_h^{i+1}(C-C_h^{i+1})\,d\x=- \frac{1}{2}\int_{\Omega}(\u - \u_h^{i+1})\cdot\big(
 \nabla C(C-C_h^{i+1})-C\nabla (C-C_h^{i+1})\big)\,d\x.
 \end{aligned}
 \end{equation}
 Using  H\"older inequality, the last term can be bounded by
 \begin{equation} \begin{aligned} \label{upbd3}
  &-\int_{\Omega}((\u - \u_h^{i+1})\cdot \nabla C)(C-C_h^{i+1})\,d\x -
    \int_{\Omega}(\u_h^{i+1}\cdot \nabla (C-C_h^{i+1})(C-C_h^{i+1})\,d\x \\
    &
     +\frac{1}{2}\int_{\Omega} \div \u_h^{i+1}C_h^{i+1}(C-C_h^{i+1})\,d\x
\leq  \frac{1}{2}(S_6^0|C|_{W^{1,3}(\Omega)}+\left\|C\right\|_{L^{\infty}(\Omega)})\left\|\u-\u_h^{i+1}\right\|_{L^2(\Omega)^d}
 |C-C_h^{i+1}|_{H^1(\Omega)}.
 \end{aligned}
 \end{equation}
 Now, the last three terms of the right-hand side of \eqref{420} can be straightforwardly bounded by
 \begin{equation} \begin{aligned} \label{425}
 \sum_{\kappa \in \mathcal{T}_h} \Big[
&\int_{\kappa}(g-g_h)(S-S_h)\,d\x
-\frac{\alpha}{2} \sum_{e \in \partial \kappa \cap \Gamma_h^i }
\int_{e} [\nabla C_h^{i+1}\cdot\n]_e (S-S_h)\, ds \\ &
+
 \int_{\kappa} (\alpha \triangle C_h^{i+1}-\u_h^{i+1}\cdot \nabla C_h^{i+1}-\frac{1}{2}\div \u_h^{i+1}C_h^{i+1}-r_0C_h^{i+1}+g_h)(S-S_h)d\x\Big] \\
 & \leq
  \sum_{\kappa \in \mathcal{T}_h} \Big[ \norm{\alpha \triangle
     C_h^{i+1}-\u_h^{i+1}\cdot \nabla C_h^{i+1}-\frac{1}{2}\div \u_h^{i+1}C_h^{i+1}-r_0C_h^{i+1}+g_h}_{L^2(\kappa)}\norm{S-S_h}_{L^2(\kappa)} \\
     &+ \left\|g-g_h\right\|_{L^2(\kappa)} \norm{S-S_h}_{L^2(\kappa)}+ \frac{1}{2}\sum_{e \in \partial \kappa \cap \Gamma_h^i}\norm{\alpha[\nabla C_h^{i+1}\cdot \n]}_{L^2(e)}\norm{S-S_h}_{L^2(e)}
     \Big].
     \end{aligned}
 \end{equation}
 Then the fact that $S_h=R_h(S)$, the approximation properties of $R_h$, equations \eqref{upbd3} and \eqref{425}, and the regularity of $\mathcal{T}_h$ yield (by using the discrete Cauchy-Schwartz inequality for equation \eqref{425})
 \begin{equation} \label{upbdconc}
 \begin{aligned}
 \alpha |C-C_h^{i+1}|_{H^1(\Omega)} &\leq
 c_1 \Big( \sum_{\kappa \in \mathcal{T}_h}
 \big( (\eta_{\kappa,i}^{(D_1)})^2 + h_{\kappa}^2\norm{g-g_h}_{L^2(\kappa)}^2 \big)\Big)^{\frac{1}{2}}
\\
& \quad + \frac{1}{2}(S_6^0|C|_{W^{1,3}(\Omega)}+\left\|C\right\|_{L^{\infty}(\Omega)})\left\|\u-\u_h^{i+1}\right\|_{L^2(\Omega)^d} \\
& \leq  c'_1 \Big( \sum_{\kappa \in \mathcal{T}_h}
 \big( \eta_{\kappa,i}^{(D_1)} + h_{\kappa}\norm{g-g_h}_{L^2(\kappa)} \big) \Big)\\
 & \qquad +\frac{1}{2}(S_6^0|C|_{W^{1,3}(\Omega)}+\left\|C\right\|_{L^{\infty}(\Omega)})\left\|\u-\u_h^{i+1}\right\|_{L^2(\Omega)^d}.
 \end{aligned}
 \end{equation}
 Next, we focus on the velocity equation. The velocity error equation \eqref{residuvitesse} can be written as
 \begin{equation} \begin{aligned} \label{upbdvit1}
 & \frac{\mu}{\rho} \int_{\Omega} K^{-1}(\u-\u_h^{i+1})\cdot \v \,d\x + \frac{\beta}{\rho} \int_{\Omega} (|\u|\u - |\u_h^{i+1}|\u_h^{i+1})\cdot \v \,d\x + \int_{\Omega} \nabla (p-p_h^{i+1})\cdot \v \,d\x \\
 & \quad = -\frac{\beta}{\rho} \int_{\Omega} \big( ( |\u_h^{i+1}|-|\u_h^i| ) \u_h^{i+1} \big) \cdot \v \,d\x + \gamma \int_{\Omega} (\u_h^{i+1}-\u_h^i)\cdot \v \,d\x \\
 & \qquad + \sum_{\kappa \in \mathcal{T}_h} \Big[
\int_{\kappa} \big( - \nabla p_h^{i+1} -\gamma (\u_h^{i+1}-\u_h^i) - \frac{\mu}{\rho} K^{-1}\u_h^{i+1}- \frac{\beta}{\rho} |\u_h^i|\u_h^{i+1}+\f_h(., C_h^i) \big) \cdot (\v - \v_h)\,d\x \\
& \qquad +\int_{\kappa} \big(\f(.,C)-\f_h(.,C)\big)\cdot \v \,d\x +
\int_{\kappa} \big( \f_h( ., C) - \f_h(. , C_h^i) \big)\cdot \v \,d\x + \int_{\kappa} \big( \f_h(. , C_h^i) - \f(. , C_h^i) \big) \cdot \v_h \, d\x \Big].
 \end{aligned}
 \end{equation}

 To simplify, we set $\z_0=\u - \u_h^{i+1}- \v_r$ and we test \eqref{upbdvit1} with $\v=\z_0 $ and $\v_h=\0$. By construction,
 \eqref{upbdvit2} and \eqref{residupression} with $q_h=r_hq$ imply that
 \begin{equation} \label{perpond} \forall q \in M, \quad \int_{\Omega} \nabla q\cdot \z_0\,d\x =0.
 \end{equation}
 Hence, \eqref{upbdvit1} becomes
 \begin{equation} \begin{aligned} \label{upbdvit5}
 \frac{\mu}{\rho} & \int_{\Omega} K^{-1}\z_0 \cdot \z_0 \,d\x +
 \frac{\mu}{\rho} \int_{\Omega}K^{-1}\v_r \cdot \z_0 \,d\x \\
 &+ \frac{\beta}{\rho} \int_{\Omega} ( |\u|\u - |\u_h^{i+1}|\u_h^{i+1}| ) ( \u - \u_h^{i+1})\,d\x
 -\frac{\beta}{\rho}\int_{\Omega} ( |\u|\u - |\u_h^{i+1}|\u_h^{i+1})\cdot \v_r \,d\x \\
 &\quad = \gamma \int_{\Omega} (\u_h^{i+1}-\u_h^i)\cdot \z_0 \,d\x -
 \frac{\beta}{\rho} \int_{\Omega} \big( ( |\u_h^{i+1}|-|\u_h^i| ) \u_h^{i+1} \big) \cdot \z_0 \,d\x \\
 & \qquad + \sum_{\kappa \in \mathcal{T}_h} \Big[
\int_{\kappa} \big( - \nabla p_h^{i+1} -\gamma (\u_h^{i+1}-\u_h^i) - \frac{\mu}{\rho} K^{-1}\u_h^{i+1}- \frac{\beta}{\rho} |\u_h^i|\u_h^{i+1}+\f_h(., C_h^i) \big) \cdot \z_0\,d\x \\
&  \qquad +\int_{\kappa} \big(\f(.,C)-\f_h(.,C)\big)\cdot \z_0 \,d\x +
\int_{\kappa} \big( \f_h( ., C) - \f_h(. , C_h^i) \big)\cdot \z_0 \,d\x  \Big].
\end{aligned}
 \end{equation}
We decompose the fourth term in the last equation as follows
\[
\ds \frac{\beta}{\rho} \int_\Omega (|\u| \u - |\u^{i+1}_h| \u^{i+1}_h) \cdot \v_r \, d\x = \ds \frac{\beta}{\rho} \int_\Omega |\u| (\u - \u^{i+1}_h) \cdot \v_r \, d\x + \frac{\beta}{\rho} \int_\Omega (|\u| - |\u^{i+1}_h|) \u^{i+1}_h \cdot \v_r \, d\x.
\]
Thus, equation \eqref{upbdvit5} becomes by inserting $\pm \u$ in the second term of the right-hand side, and by using the monotonicity of $\mathcal{A}$
\begin{equation} \begin{aligned} \label{upbdvit6}
& \frac{\mu}{\rho} \int_{\Omega} K^{-1}\z_0 \cdot \z_0 \,d\x +c_m\left\|\u -\u_h^{i+1}\right\|_{L^3(\Omega)^d}^3
\leq \\
& \frac{\mu}{\rho} \int_{\Omega}|K^{-1} \v_r| |\z_0 |\,d\x + \frac{\beta}{\rho}
\int_{\Omega}|\u| |\u - \u_h^{i+1}||\v_r|\,d\x + \frac{\beta}{\rho}
\int_{\Omega}|\u - \u_h^{i+1}||\u_h^{i+1}||\v_r|\,d\x \\
&+ \gamma \int_{\Omega}|\u - \u_h^{i+1}||\z_0|\,d\x + \frac{\beta}{\rho}\int_{\Omega}|\u_h^{i+1}-\u_h^i||\u_h^{i+1}-\u||\z_0|\,d\x + \frac{\beta}{\rho}\int_{\Omega}|\u_h^{i+1}-\u_h^i||\u||\z_0|\,d\x \\
& + \sum_{\kappa \in \mathcal{T}_h} \Big[
\int_{\kappa} |- \nabla p_h^{i+1} -\gamma (\u_h^{i+1}-\u_h^i) - \frac{\mu}{\rho} K^{-1}\u_h^{i+1}- \frac{\beta}{\rho} |\u_h^i|\u_h^{i+1}+\f_h(., C_h^i)| |\z_0| \,d\x \\
&+ \int_{\kappa} \big|\f(.,C)-\f_h(.,C)| | \z_0 |\,d\x +
\int_{\kappa} \big| \f_h( ., C) - \f_h(. , C_h^i) \big|| \z_0 |\,d\x  \Big].
\end{aligned}
\end{equation}
By using the relation $\u-\u_h^{i+1}=\z_0 + \v_r$ and applying \eqref{fhlipshitz}, we obtain
\begin{equation}
\begin{aligned} \label{upbdvit7}
& \frac{\mu K_m}{\rho} \| \textbf{z}_0\|^2_{L^2(\Omega)^d}+c_m\left\|\u -\u_h^{i+1}\right\|_{L^3(\Omega)^d}^3 \le \\
& \frac{\mu K_M}{\rho} \| \textbf{v}_r \|_{L^2(\Omega)^d} \, \| \textbf{z}_0 \|_{L^2(\Omega)^d} +  \frac{\beta}{\rho} \| \u \|_{L^6(\Omega)^d} (\| \textbf{z}_0 \|_{L^2(\Omega)^d} + \| \v_r \|_{L^2(\Omega)^d}) \| \v_r \|_{L^3(\Omega)^d} \\
& + \ds \frac{\beta}{\rho} (\| \v_r \|_{L^2(\Omega)^d}  + \| \textbf{z}_0 \|_{L^2(\Omega)^d})
 \|\u^{i+1}_h \|_{L^6(\Omega)^d} \, \| \v_r \|_{L^3(\Omega)^d} \ds + \ds \gamma \| \u_h^{i+1} - \u_h^{i} \|_{L^2(\Omega)^d} \, \| \textbf{z}_0 \|_{L^2(\Omega)^d} \\
 & +  \frac{\beta}{\rho} \| \u^{i+1}_h-\u^i_h \|_{L^\infty (\Omega)^d}   \,
 (\| \v_r \|_{L^2(\Omega)^d}  + \| \textbf{z}_0 \|_{L^2(\Omega)^d})
  \,  \|\textbf{z}_0\|_{L^2(\Omega)^d}  + \frac{\beta}{\rho} \| \u^{i+1}_h-\u^i_h \|_{L^2(\Omega)^d}   \, \|\u \|_{L^\infty (\Omega)^d} \,  \|\textbf{z}_0\|_{L^2(\Omega)^d}\\
  &+ \sum_{\kappa \in  \mathcal{T}_h } \|-\nabla p_h^{i+1}-{\frac{\mu}{\rho}} K^{-1} \u^{i+1}_h - \gamma (\u_h^{i+1} - \u_h^{i}) - \frac{\beta}{\rho} |\u^i_h| \u^{i+1}_h + \f_h(., C_h^i) \|_{L^2(\kappa)} \, \|\textbf{z}_0\|_{L^2(\kappa)}\\
  & + \sum_{\kappa \in \mathcal{T}_h}\|\f(.,C)-\f_h(.,C)\|_{L^2(\kappa)}\|\z_0\|_{L^2(\kappa)} + \sum_{\kappa \in \mathcal{T}_h} c_{\f_1}  \|C-C_h^i\|_{L^2(\kappa)}\|\z_0\|_{L^2(\kappa)}.
\end{aligned}
\end{equation}
Lemma \ref{uborne6} ensures that there exists an integer $i_0$ depending on $h$ such that $\forall i \geq i_0$, $\u_h^{i+1}$ is bounded in $L^6(\Omega)^d$. Furthermore, we shall use the following inverse inequality
\[
\|\u_h^{i+1}-\u_h^i\|_{L^{\infty}(\Omega)^d} \leq c_I h^{-d/2}
\|\u_h^{i+1}-\u_h^i\|_{L^{2}(\Omega)^d},
\]
and the convergence of the sequence $\u_h^i$ to $\u_h$ in $L^2(\Omega)^d$ to deduce that we can chose $i_0$ sufficiently large so that $ \ds \|\u_h^{i+1}-\u_h^i\|_{L^{2}(\Omega)^d} \leq h^{d/2} \frac{\mu K_m}{2\beta c_I}$, and then $ \ds \left\|\u_h^{i+1}-\u_h^{i}\right\|_{L^2(\Omega)^d} \leq \frac{\mu K_m}{2 \beta}$. Therefore, equation \eqref{upbdvit7}  becomes by inserting $C_h^{i+1}$ to the last term
\begin{equation} \begin{aligned} \label{upbdvit8}
& \frac{\mu K_m}{2\rho} \| \textbf{z}_0\|^2_{L^2(\Omega)^d} +c_m\left\|\u -\u_h^{i+1}\right\|_{L^3(\Omega)^d}^3\le \\
& \frac{\mu K_M}{\rho} \| \textbf{v}_r \|_{L^2(\Omega)^d} \, \| \textbf{z}_0 \|_{L^2(\Omega)^d} +  \frac{\beta}{\rho} \| \u \|_{L^6(\Omega)^d} (\| \textbf{z}_0 \|_{L^2(\Omega)^d} + \| \v_r \|_{L^2(\Omega)^d}) \| \v_r \|_{L^3(\Omega)^d} \\
& + \ds \frac{\beta}{\rho} \hat{c_1}(\u,p,C) (\| \v_r \|_{L^2(\Omega)^d}  + \| \textbf{z}_0 \|_{L^2(\Omega)^d})
  \| \v_r \|_{L^3(\Omega)^d} \ds + \ds \gamma \| \u_h^{i+1} - \u_h^{i} \|_{L^2(\Omega)^d} \, \| \textbf{z}_0 \|_{L^2(\Omega)^d} \\
 &+ \frac{\mu K_m}{2\rho} \, \| \v_r \|_{L^2(\Omega)^d}
  \,  \|\textbf{z}_0\|_{L^2(\Omega)^d}
 + \frac{\beta}{\rho} \| \u^{i+1}_h-\u^i_h \|_{L^2(\Omega)^d}   \, \|\u \|_{L^\infty(\Omega)^d}  \,  \|\textbf{z}_0 \|_{L^2(\Omega)^d} \\
 &+ \ds \sum_{\kappa \in  \mathcal{T}_h } \|-\nabla p_h^{i+1}- \frac{\mu}{\rho}
 K^{-1} \u^{i+1}_h - \gamma (\u_h^{i+1} - \u_h^{i}) -\frac{\beta}{\rho} |\u^i_h| \u^{i+1}_h + \f_h(.,C_h^i) \|_{L^2(\kappa)} \, \|\textbf{z}_0\|_{L^2(\kappa)}\\
 &+   \ds \sum_{\kappa \in  \mathcal{T}_h } \| \f(.,C)-\f_h(.,C) \|_{L^2(\kappa)}  \, \|\textbf{z}_0\|_{L^2(\kappa)} +
 \ds \sum_{\kappa \in  \mathcal{T}_h } c_{\f_1}\|C_h^{i+1}-C_h^i\|_{L^2(\kappa)}\|\z_0\|_{L^2(\kappa)} \\
 &+\ds \sum_{\kappa \in  \mathcal{T}_h } c_{\f_1}\|C-C_h^{i+1}\|_{L^2(\kappa)}\|\z_0\|_{L^2(\kappa)} .
\end{aligned} \end{equation}
We use the decomposition $ \ds ab \leq \frac{1}{2\varepsilon}a^2 + \frac{\varepsilon}{2}b^2$ for all the terms containing $\|\z_0\|_{L^2(\Omega)^d}$ in the right-hand side of the previous equation with $\varepsilon$ sufficiently small so that all the terms of $\|\z_0\|_{L^2(\Omega)^d}$ will be dominated  by the left-hand side. Then by using the inequality $\ds \|\v_r\|_{L^2(\Omega)^d} \leq |\Omega|^{1/6} \|\v_r\|_{L^3(\Omega)^d}$, 
the regularity of $\mathcal{T}_h$, and by taking the square root of the inequality, the following bound holds
\begin{equation} \label{util11}
    \begin{aligned}
    \norm{\z_0}_{L^2(\Omega)^d}&+c_m\left\|\u -\u_h^{i+1}\right\|_{L^3(\Omega)^d}^{3/2} \leq
    \Tilde{c} \Big(  \norm{\v_r}_{L^3(\Omega)^d} + \sum_{\kappa \in \mathcal{T}_h} \norm{\u_h^{i+1}-\u_h^i}_{L^2(\kappa)} \\
     +  \sum_{\kappa \in \mathcal{T}_h} \norm{\f(.,C)-\f_h(., C)}_{L^2(\kappa)}+&\hspace{-2mm}\sum_{\kappa \in \mathcal{T}_h}\hspace{-1mm}\|\nabla p_h^{i+1}+ \frac{\mu}{\rho}
 K^{-1} \u^{i+1}_h + \gamma (\u_h^{i+1} - \u_h^{i}) +\hspace{-1mm}\frac{\beta}{\rho} |\u^i_h| \u^{i+1}_h - \f_h(.,C_h^i) \|_{L^2(\kappa)} \\
 & + \sum_{\kappa \in \mathcal{T}_h}||C_h^{i+1}-C_h^i||_{H^1(\kappa)}
    \Big) + \frac{8\rho c_{\f_1}^2}{\mu K_m}S_2^0 |C- C_h^{i+1}|_{H^1(\Omega)},
    \end{aligned}
\end{equation}
where $\tilde{c}$ is a  constant depending on the exact solution.\\
\noindent Thus, we deduce from the relation $\ds \u - \u_h^{i+1}= \z_0 + \v_r$, the triangle inequality $\ds \norm{\u-\u_h^{i+1}}_{L^2(\Om)^d} \leq \norm{\u-\u_h^{i+1}-\v_r}_{L^2(\Om)^d} + \norm{\v_r}_{L^2(\Om)^d} $, using relations \eqref{util11} and \eqref{upbdvit4}, and the following inequality:
\begin{equation} \begin{aligned} \label{tildecut}
\norm{\u - \u_h^{i+1}}_{L^2(\Om)^d} & \leq \Tilde{c} \Big( \sum_{\kappa \in \mathcal{T}_h} \big(
\eta_{\kappa , i}^{(D_3)}+\eta_{\kappa , i}^{(L_1)}+\eta_{\kappa , i}^{(L_2)}+\eta_{\kappa , i}^{(D_2)} + \norm{\f(.,C)-\f_h(.,C)}_{L^2(\kappa)}\big) \Big) \\
&+ \frac{8\rho c_{\f_1}^2}{\mu K_m}S_2^0  |C-C_h^{i+1}|_{H^1(\Omega)}.
\end{aligned}
\end{equation}
When substituted into \eqref{upbdconc}, this estimate for the velocity error gives
\begin{equation} \begin{aligned} \label{concent9}
\Big(\alpha&- \big( S_6^0 |C|_{W^{1,3}(\Omega)}+ \norm{C}_{L^{\infty}(\Omega)}\big)\frac{4\rho c_{\f_1}^2}{\mu K_m}S_2^0  \Big) |C-C_h^{i+1}|_{H^1(\Omega)}
 \leq \Tilde{c_2} \Big( \sum_{\kappa \in \mathcal{T}_h} \big( \eta_{\kappa ,i}^{(D_1)}+\eta_{\kappa ,i}^{(D_2)}+\eta_{\kappa ,i}^{(D_3)}\\
 &+ \eta_{\kappa ,i}^{(L_1)}+\eta_{\kappa ,i}^{(L_2)}+ h_k \norm{g-g_h}_{L^2(\kappa)}+\norm{\f(.,C)-\f_h(.,C)}_{L^2(\kappa)}\big) \Big).
 \end{aligned} \end{equation}
 In view of \eqref{conditionC}, the concentration estimate in \eqref{upperbound1} follows from \eqref{concent9}, and in turn, the velocity estimate follows by substituting \eqref{concent9} in \eqref{tildecut}.

$\hfill\Box$

 \end{proof}
 {\rmq Similarly to Theorem 3.12 in \cite{semaan2}, the last theorem gives an upper bound for the error $\u - \u_h^{i+1}$ in $L^2(\Omega)^d$. But unfortunately, it gives an upper bound of $\| \u - \u_h^{i+1} \|_{L^3(\Omega)^d}$ with the indicators to the power of $2/3$. \\}
 \begin{thm}
 We retain the assumptions of Theorem \ref{thmupbd1}. There exists a positive real number $i_1$ depending on $h$ such that $\forall i \geq i_1$, we have the following bound
 \begin{equation} \begin{aligned} \label{upperbound2}
\norm{\nabla(p-p_h^{i+1})}_{L^{3/2}(\Omega)^d}&\leq
\Tilde{c_1} \Big( \sum_{\kappa \in \mathcal{T}_h} \big( \eta_{\kappa ,i}^{(D_1)}+\eta_{\kappa ,i}^{(D_2)}+\eta_{\kappa ,i}^{(D_3)}
 + \eta_{\kappa ,i}^{(L_1)}\\
 &+\eta_{\kappa ,i}^{(L_2)}
 + h_k \norm{g-g_h}_{L^2(\kappa)}+\norm{\f(.,C)-\f_h(.,C)}_{L^2(\kappa)}\big) \Big).
 \end{aligned}
\end{equation}
 \end{thm}
 \begin{proof}
 Let $(\u,p,C)$ and $(\u_h^{i+1},p_h^{i+1},C_h^{i+1})$ be the respective solutions of $(V_a)$ and $(V_{ahi})$. We test equation \eqref{residuvitesse} with $\v_h=\0$ to get
 \begin{equation} \label{residuvit3}
 \begin{aligned}
 \int_{\Omega} \nabla (p-p_h^{i+1})\cdot \v \,d\x
 &=-\frac{\mu}{\rho} \int_{\Omega} K^{-1}(\u-\u_h^{i+1})\cdot \v \,d\x - \frac{\beta}{\rho} \int_{\Omega} (|\u|\u - |\u_h^{i}|\u_h^{i+1})\cdot \v \,d\x\\
 &+\sum_{\kappa \in \mathcal{T}_h} \Big[
\int_{\kappa} \big( - \nabla p_h^{i+1} -\gamma (\u_h^{i+1}-\u_h^i) - \frac{\mu}{\rho} K^{-1}\u_h^{i+1}- \frac{\beta}{\rho} |\u_h^i|\u_h^{i+1}+\f_h(., C_h^i) \big) \cdot \v\,d\x \\
&
+ \gamma \int_{\kappa } (\u_h^{i+1}-\u_h^i)\cdot \v \,d\x + \int_{\kappa} \big(\f(.,C)-\f_h(.,C)\big)\cdot \v \,d\x\\
&+
\int_{\kappa} \big( \f_h( ., C) - \f_h(. , C_h^i) \big)\cdot \v \,d\x
\Big],
\end{aligned}
 \end{equation}
 By using the Cauchy-Shwartz inequality and dividing by $\norm{\v}_{L^3(\Omega)^d}$, we get
 \begin{equation}\label{equa1}
\begin{array}{ll}
\medskip
\ds \frac{\ds \Big| \int_\Omega \nabla (p_h^{i+1} - p) \v \, d\x \Big| }{||\v ||_{L^3(\Omega)^d}} \le  c \ds ( || \u - \u_h^{i+1}||_{L^2(\Omega)^d} +  || \u_h^i - \u_h^{i+1}||_{L^2(\Omega)^d}) \frac{||\v||_{L^2(\Omega)^d}}{||\v||_{L^3(\Omega)^d}} \\
 \ds  + c_1 \Big( \sum_{\kappa \in \mathcal{T}_h} ||\f(.,C) - \f_h(.C)||^2_{L^2(\kappa)} \Big)^{1/2}
 \frac{||\v||_{L^2(\Omega)^d}}{||\v||_{L^3(\Omega)^d}} + \frac{\beta}{\rho} \Big| \int_\Omega (|\u| \u - |\u^i_h| \u^{i+1}_h) \cdot \v \, d\x   \Big| \frac{1}{||\v||_{L^3(\Omega)^d}}\\
 + \ds c_2 \Big( \sum_{\kappa \in \mathcal{T}_h} || -\nabla p_h^{i+1}- \gamma (\u_h^i - \u_h^{i+1}) - \frac{\mu}{\rho}K^{-1} \u^{i+1}_h - \frac{\beta}{\rho} |\u^i_h| \u^{i+1}_h + \f_h(.,C_h^i) ||^2_{L^2(\kappa)}  \Big)^{1/2} \frac{||\v||_{L^2(\Omega)^d}}{||\v||_{L^3(\Omega)^d}}\\
 +\ds c_3 \Big(\sum_{\kappa \in \Gamma_h} \norm{C-C_h^i}_{L^2(\kappa)}^2\Big)^{1/2}\frac{||\v||_{L^2(\Omega)^d}}{||\v||_{L^3(\Omega)^d}}.
\end{array}
\end{equation}
By using the relation $||\v||_{L^2(\Omega)^d} \le |\Omega |^{1/6} ||\v||_{L^3(\Omega)^d}$, all the terms of the right hand side of the previous bound can be treated as in the previous theorem except the third one which can be bounded as following:
\begin{equation}\label{equa2}
\begin{array}{rcl}
\Big| ( |\u| \u - |\u^i_h| \u^{i+1}_h, \v )  \Big|  &\le& \Big|  ( (|\u| - |\u_h^i |) \u, \v)  \Big| +
\Big| (|\u_h^i| (\u - \u_h^{i+1}) , \v)   \Big|\\
&\le& \big(   ||\u - \u^i_h||_{L^2(\Omega)^d} ||\u ||_{L^6(\Omega)^d} + ||\u^i_h||_{L^6(\Omega)^d} ||\u - \u_h^{i+1} ||_{L^2(\Omega)^d}   \big)||\v ||_{L^3(\Omega)^d}.
\end{array}
\end{equation}
We consider equation \eqref{equa1}. By using the following inequality
\[
\norm{\u-\u_h^{i}}_{L^2(\Omega)^d} \leq \norm {\u-\u_h^{i+1}}_{L^2(\Omega)^d}+\norm{\u_h^{i+1}-\u_h^i}_{L^2(\Omega)^d},
\]
the fact that $\norm{\u_h^i}_{L^6(\Omega)^d}$ is bounded, the inf-sup condition \eqref{infsup} and Theorem \ref{thmupbd1}, we get the desired error bound on the pressure.
 \end{proof}
 \begin{Rem}
 The bounds \eqref{upperbound1} and \eqref{upperbound2} constitute our {\it a posteriori} error estimates where we bound the error between the exact solution $(\u,p,C)$ of $(V_a)$ and the numerical solution $(\u_h^{i+1},p_h^{i+1},C_h^{i+1})$ of $(V_{ahi})$ with respect to the indicators $\eta_{\kappa,i}^{(L)}$, $\eta_{\kappa,i}^{(D_1)}$, $\eta_{\kappa,i}^{(D_2)}$ and $\eta_{\kappa,i}^{(D_3)}$.
 But to get the bounds of the indicators which are the subject of the next subsection (Section \ref{optimall}), we need to add the following theorem where we add a supplementary bound giving an error bound of the exact and numerical solutions.
 \end{Rem}
\begin{thm}
Under the assumptions of Lemma \ref{uborne6} and if $\f_0 \in L^2(\Omega)^d$, there exists an integer $i_0$ depending on $h$ such that for all $i\ge i_0$,  the solutions $(\u,p)$ of $(V_a)$ and $(\u_h^{i+1},p_h^{i+1})$ of $(V_{ahi})$ verify the following error inequalities:
\begin{equation} \begin{aligned} \label{upperbound4}
\norm{\big(\frac{\beta}{\rho}(|\u|\u - |\u_h^{i}|\u_h^{i+1})+\nabla (p-p_h^{i+1})\big)}_{L^2(\Omega)^d}&\leq
\Tilde{c_2} \Big( \sum_{\kappa \in \mathcal{T}_h} \big( \eta_{\kappa ,i}^{(D_1)}+\eta_{\kappa ,i}^{(D_2)}+\eta_{\kappa ,i}^{(D_3)}
 + \eta_{\kappa ,i}^{(L_1)}\\
 &+\eta_{\kappa ,i}^{(L_2)}
 + h_k \norm{g-g_h}_{L^2(\kappa)}+\norm{\f(.,C)-\f_h(.,C)}_{L^2(\kappa)}\big) \Big)
 \end{aligned}
\end{equation}
where $\Tilde{c_2}$ is a constant independent of $h$.
\end{thm}
\begin{proof}
Using the fact that $\u \in L^{\infty}(\Omega)^d$, then the first equation of system $(P)$ allows us to get $\nabla p \in L^2(\Omega)^d$. Thus, the velocity error equation \eqref{residuvitesse} is valid for all $\v$ in $L^2(\Omega)^d$ and can be written as:
\begin{equation} \label{residuvitesse1}
\begin{aligned}
& \int_{\Omega} \big(\frac{\beta}{\rho}(|\u|\u - |\u_h^{i}|\u_h^{i+1})+\nabla (p-p_h^{i+1})\big)\cdot \v \,d\x=-\frac{\mu}{\rho} \int_{\Omega} K^{-1}(\u-\u_h^{i+1})\cdot \v \,d\x \\
& +\sum_{\kappa \in \mathcal{T}_h} \Big[
\int_{\kappa} \big( - \nabla p_h^{i+1} -\gamma (\u_h^{i+1}-\u_h^i) - \frac{\mu}{\rho} K^{-1}\u_h^{i+1}- \frac{\beta}{\rho} |\u_h^i|\u_h^{i+1}+\f_h(., C_h^i) \big) \cdot (\v - \v_h)\,d\x \\
&
+ \gamma \int_{\kappa } (\u_h^{i+1}-\u_h^i)\cdot \v \,d\x + \int_{\kappa} \big(\f(.,C)-\f_h(.,C)\big)\cdot \v \,d\x +
\int_{\kappa} \big( \f_h( ., C) - \f_h(. , C_h^i) \big)\cdot \v \,d\x
\\
&
+ \int_{\kappa} \big( \f_h(. , C_h^i) - \f(. , C_h^i) \big) \cdot \v_h \, d\x
\Big],
\end{aligned}
\end{equation}
By taking $\v_h=\0$ and $\ds \v=\big(\frac{\beta}{\rho}(|\u|\u - |\u_h^{i}|\u_h^{i+1})+\nabla (p-p_h^{i+1})\big)$, applying the Cauchy-Schwartz inequality and simplifying by $\left\|\v\right\|_{L^2(\Omega)^d}$, we get the result by applying Theorem \ref{thmupbd1}.
\end{proof}
\subsubsection{Bounds of the indicators}\label{optimall}
\noindent In order to establish the efficiency of the {\it a posteriori} error estimates, we recall the following properties (see R. Verf\"urth,\cite{Verfurth2013}, Chapter 1). For an element $\kappa$ of $\mathcal{T}_h$, we consider the bubble function $\psi_\kappa$ (resp. $\psi_e$ for the face $e$) which is equal to the product of the $d+1$ barycentric coordinates associated with the vertices of $\kappa$ (resp. of the $d$ barycentric coordinates associated with the vertices of $e$). We also consider a lifting operator ${\mathcal{L}}_{e}$ defined on polynomials on $e$ vanishing on $\partial e$ into polynomials on the at most two elements $\kappa$ containing $e$ and vanishing on $\partial \kappa \setminus e $, which is constructed by affine transformation from a fixed operator on the reference element.
\begin{propri}\label{psi}
 Denoting by $Pr(\kappa)$ the space of polynomials of degree smaller than $r$ on $\kappa$. The following properties hold:
\begin{equation}
\forall v \in P_r(\kappa), \qquad
\begin{cases}
c ||v||_{0,\kappa} \le ||v \psi^{1/2}_{\kappa} ||_{0,\kappa}
\le c' ||v||_{0,\kappa}, &\\
|v|_{1,\kappa} \le c h_{\kappa}^{-1} ||v ||_{0,\kappa}.&
\end{cases}
\end{equation}
\end{propri}
\begin{propri} \label{psii} Denoting by $Pr(e)$ the space of polynomials of degree smaller than $r$ on $e$, we have
$$\forall\; v  \in P_r(e),\qquad
c\Vert v \Vert_{0,e}\leq \Vert v\psi_{e}^{1/2}
\Vert_{0,e}\leq c'\Vert  v \Vert_{0,e},$$
and, for all polynomials $v$ in $Pr(e)$ vanishing on $\partial e$, if $\kappa$ is an element which contains $e$,
$$ \Vert {\mathcal{L}}_{e}v \Vert_{0,\kappa}+h_{e}\mid
{\mathcal{L}}_{e}v \mid_{1,\kappa}\leq ch^{1/2}_{e}\Vert  v
\Vert_{0,e}.$$
\end{propri}
Let us start with the concentration errors.
\begin{thm} \label{indicateurc}
Under the assumptions of  Theorem \ref{thm42} and Lemma \ref{uborne6} , for all $\kappa \in \mathcal{T}_h$ we have:
\begin{equation} \label{etaid1}
\eta_{\kappa , i}^{(D_1)} \leq c \big( \left\|\u_h^{i+1}-\u\right\|_{L^2(w_e)} + \left\|C-C_h^{i+1}\right\|_{H^1(w_e)} + \sum _{\kappa \subset w_e}h_{\kappa} \left\|g-g_h\right\|_{L^2(\kappa)} \big)
\end{equation}
where $c$ is a positive constant depending on the exact solution but independent of $h$ and $w_k$.  Moreover, without any assumption we have:
 \begin{equation} \label{etail2}
\eta_{\kappa,i}^{(L_2)} \leq \left\|C-C_h^{i+1}\right\|_{H^1(\kappa)} + \left\|C-C_h^{i}\right\|_{H^1(\kappa)} .
 \end{equation}
\end{thm}
 \begin{proof}
 The bound \eqref{etail2} is obvious using a triangular inequality.
 In order to derive a lower bound for the interior part of $\eta_{\kappa , i}^{(D_1)}$, we start from equation \eqref{residuconcent}. By using that $\div \u=0$ and applying Green's formula, the term
 $\ds -\frac{1}{2}\int_{\Omega}\div \u_h^{i+1}C_h^{i+1}S\,d\x$ can be written as:
 \begin{equation} \begin{aligned} \label{4.44}
 -\frac{1}{2}\int_{\Omega}\div \u_h^{i+1}C_h^{i+1}S\,d\x &=
 -\frac{1}{2}\int_{\Omega}\div (\u_h^{i+1} -\u)C_h^{i+1}S\,d\x \\
 &= -\frac{1}{2}\int_{\Omega}\div (\u_h^{i+1} -\u)(C_h^{i+1}- R_h(C))S\,d\x  -\frac{1}{2}\int_{\Omega}\div (\u_h^{i+1} -\u)R_h(C)S\,d\x \\
 &= \frac{1}{2}\int_{\Omega} (\u_h^{i+1}-\u)\cdot\nabla (C_h^{i+1}-R_h(C))S\,d\x + \frac{1}{2}\int_{\Omega} (\u_h^{i+1}-\u)\cdot\nabla S (C_h^{i+1}-R_h(C))\,d\x\\
 &+\frac{1}{2}\int_{\Omega}(\u_h^{i+1}-\u)\cdot ((\nabla R_h(C))S+(\nabla S) R_h(C))\,d\x.
 \end{aligned} \end{equation}
 In view of \eqref{4.44}, by taking $S_h=0$ and $S=S_{\kappa}$, where in each element $\kappa$, $S_{\kappa}$ is the localizing function
 \begin{equation} \label{skappa}
 S_{\kappa} = \big( \alpha \triangle C_h^{i+1}-\u_h^{i+1}\cdot \nabla C_h^{i+1}-\frac{1}{2}\div \u_h^{i+1}C_h^{i+1}-r_0C_h^{i+1}+g_h \big)\psi_{\kappa}
 \end{equation}
 extended by $0$ outside $\kappa$, hence, equation \eqref{residuconcent} becomes:
 \begin{equation}
 \begin{aligned}
 \label{residuconcent1}
 &\int_{\kappa} (\alpha \triangle C_h^{i+1}-\u_h^{i+1}\cdot \nabla C_h^{i+1}-\frac{1}{2}\div \u_h^{i+1}C_h^{i+1}-r_0C_h^{i+1}+g_h)^2 \psi_{\kappa}\,d\x= \alpha \int_{\kappa}\nabla (C-C_h^{i+1})\cdot \nabla S_{\kappa}\,d\x \\
 &+ r_0\int_{\kappa}(C-C_{h}^{i+1})S_{\kappa}\,d\x+ \int_{\kappa}((\u-\u_{h}^{i+1})\cdot \nabla C)S_{\kappa}\,d\x +
 \int_{\kappa}(\u_h^{i+1}\cdot \nabla (C-C_h^{i+1}))S_{\kappa}\,d\x \\
 & - \int_{\kappa} (g-g_h)S_{\kappa}\,d\x +\frac{1}{2}\int_{\kappa} (\u_h^{i+1}-\u)\cdot\nabla (C_h^{i+1}-R_h(C))S_{\kappa}\,d\x + \frac{1}{2}\int_{\kappa} (\u_h^{i+1}-\u)\cdot\nabla S_{\kappa} (C_h^{i+1}-R_h(C))\,d\x \\
 & +\frac{1}{2}\int_{\kappa}(\u_h^{i+1}-\u)\cdot ((\nabla R_h(C))S_{\kappa}+(\nabla S_{\kappa}) R_h(C))\,d\x.
 \end{aligned}
 \end{equation}
 In order to bound the right-hand-side of the previous equation, we start by bounding the last three terms as following:
 \begin{equation}
     \begin{aligned}
     \label{447}
     &\frac{1}{2}\int_{\kappa} (\u_h^{i+1}-\u)\cdot\nabla (C_h^{i+1}-R_h(C))S_{\kappa}\,d\x + \frac{1}{2}\int_{\kappa} (\u_h^{i+1}-\u)\cdot\nabla S_{\kappa} (C_h^{i+1}-R_h(C))\,d\x \\
 & \qquad +\frac{1}{2}\int_{\kappa}(\u_h^{i+1}-\u)\cdot ((\nabla R_h(C))S_{\kappa}+(\nabla S_{\kappa}) R_h(C))\,d\x \\
    & \ds  \leq \frac{1}{2}\left\|\u-\u_h^{i+1}\right\|_{L^2(\kappa)} \big(
     c_I(6)c_I(3)h_{\kappa}^{-d/2}|C_h^{i+1}-R_h(C)|_{H^1(\kappa)} +c_I(3)c_L h_{\kappa}^{(-d-6)/6}\left\|C_h^{i+1}-R_h(C)\right\|_{L^6(\kappa)} \\
     & \qquad \ds + c_I(6)h_{\kappa}^{-d/3}|R_h(C)|_{W^{1,3}(\kappa)} + c_L h_{\kappa}^{-1}\left\|R_h(C)\right\|_{L^{\infty}(\kappa)}
     \big) \left\|S_{\kappa}\right\|_{L^2(\kappa)}
     \\
     & \ds \leq \frac{1}{2}\left\|\u-\u_h^{i+1}\right\|_{L^2(\kappa)} \big(
     c_I(6)c_I(3)h_{\kappa}^{-d/2} ( |C_h^{i+1}-C_h|_{H^1(\kappa)} + |C_h-C|_{H^1(\kappa)} + |R_h(C)-C|_{H^1(\kappa)} ) \\
     & \ds \qquad +c_I(3)c_L h_{\kappa}^{(-6-d)/6} ( \left\|C_h^{i+1}-C_h\right\|_{L^6(\kappa)} + \left\|C_h- C\right\|_{L^6(\kappa)} + \left\|R_h(C)-C\right\|_{L^6(\kappa)} )
     \\
     & \qquad +\ds  c_I(6)h_{\kappa}^{-d/3}|R_h(C)|_{W^{1,3}(\kappa)} + c_L h_{\kappa}^{-1}\left\|R_h(C)\right\|_{L^{\infty}(\kappa)}
     \big) \left\|S_{\kappa}\right\|_{L^2(\kappa)}.
     \end{aligned}
 \end{equation}
 Following Theorem \ref{thm42} and as the sequence $C_h^{i+1}$ converges strongly to $C_h$ in $H^1(\Omega)$ so there exists an integer $i_1$ depending on $h$ such that for all $i\geq i_1$, we have
 \begin{equation} \label{converg} |C_h^{i+1}-C_h|_{H^1(\kappa)} \leq h  \qquad \text{ and } \qquad \left\|C_h^{i+1}-C_h\right\|_{L^6(\kappa)} \leq S_6^0h.
 \end{equation}
 Thus, using \eqref{converg}, the {\it{a priori}} error estimates, the regularity properties of the operator $R_h$ and the fact that the mesh is uniformly regular, we get:
 \begin{equation}
     \begin{aligned}
     \label{448}
     &\frac{1}{2}\int_{\kappa} (\u_h^{i+1}-\u)\cdot\nabla (C_h^{i+1}-R_h(C))S_{\kappa}\,d\x + \frac{1}{2}\int_{\kappa} (\u_h^{i+1}-\u)\cdot\nabla S_{\kappa} (C_h^{i+1}-R_h(C))\,d\x \\
 & +\frac{1}{2}\int_{\kappa}(\u_h^{i+1}-\u)\cdot ((\nabla R_h(C))S_{\kappa}+(\nabla S_{\kappa}) R_h(C))\,d\x \\
    &  \leq \tilde{c}_1\left\|\u-\u_h^{i+1}\right\|_{L^2(\kappa)} \big( h_{\kappa}^{(2-d)/2} + h_{\kappa}^{-d/6} + h_{\kappa}^{-d/3} + h_{\kappa}^{-1} \big) \left\|S_{\kappa}\right\|_{L^2(\kappa)}
    \end{aligned} \end{equation}
    where $\tilde{c}_1$ is a positive constant depending on the exact solution but independent of $h$. \\
 In view of \eqref{448}, using Holder inequality ,the inverse inequalities and Lemma \ref{uborne6}, the left-hand-side of \eqref{residuconcent1} can be bounded as follows:
 \begin{equation}
     \begin{aligned} \label{etaid11}
     & \int_{\kappa} (\alpha \triangle C_h^{i+1}-\u_h^{i+1}\cdot \nabla C_h^{i+1}-\frac{1}{2}\div \u_h^{i+1}C_h^{i+1}-r_0C_h^{i+1}+g_h)^2 \psi_{\kappa}\,d\x \\
     & \leq \big[\alpha c_L h_{\kappa}^{-1} \left\|C-C_h^{i+1}\right\|_{H^1(\kappa)}+ r_0 \left\|C-C_h^{i+1}\right\|_{H^1(\kappa)} + c_I(6)h_{\kappa}^{-d/3}\left\|\u -\u_h^{i+1}\right\|_{L^2(\kappa)}|\nabla C|_{W^{1,3}(\kappa)} \\
     &+ \hat{c}(\u,p,C) c_I(3) h_{\kappa}^{-d/6}\left\|C-C_h^{i+1}\right\|_{H^1(\kappa)}
     + \left\|g-g_h\right\|_{L^2(\kappa)} \\
     &+ \tilde{c}_1( h_{\kappa}^{(2-d)/2} + h_{\kappa}^{-d/6} + h_{\kappa}^{-d/3} + h_{\kappa}^{-1}
     )
     \left\|\u-\u_h^{i+1}\right\|_{L^2(\kappa)}
     \big] \left\|S_{\kappa}\right\|_{L^2(\kappa)}.
     \end{aligned}
 \end{equation}
 Finally, we get the bound for the first part of $\eta_{\kappa , i}^{(D_1)} $ by using Property \ref{psi} and multiplying the previous inequality by $h_{\kappa}$. \\
 \noindent Finally, we estimate the surface part of $\eta_{\kappa , i}^{(D_1)} $ by testing \eqref{residuconcent} with $S_h=0$ and $S=S_e$ where $S_e$ is the localizing function defined by
 \begin{equation*}
S_{e}=
\left \{
\begin{array}{lcl}
\mathcal{L}_{e} \big( \alpha [\nabla C_h^{i+1} \cdot \n]_{e} \psi_e
  \big)& \hspace{-0.5cm} \mbox{on }  \kappa \cup \kappa' , \\
 & \hspace{-0.3cm}\mbox{ on } \Omega \backslash  (\kappa \cup \kappa') , \\
\end{array}
 \right.
 \end{equation*}
and $\kappa$ and $\kappa '$ are the two elements adjacent to $e$. Then \eqref{residuconcent} reduces to
\begin{equation} \begin{aligned} \label{residuconc3}
&\alpha \int_{e}[\nabla C_{h}^{i+1}\cdot \n]_e^2 \psi_e \,ds \\
&=\int_{\kappa \cup \kappa '} (\alpha \triangle C_h^{i+1}-\u_h^{i+1}\cdot \nabla C_h^{i+1}-\frac{1}{2}\div \u_h^{i+1}C_h^{i+1}-r_0C_h^{i+1}+g_h) S_e\,d\x + \int_{\kappa \cup \kappa '} (g-g_h)S_e\,d\x \\
&-r_0 \int_{\kappa \cup \kappa '}(C-C_h^{i+1})S_e\,d\x - \alpha \int_{\kappa \cup \kappa '} \nabla (C-C_h^{i+1})\cdot\nabla S_e\,d\x + \frac{1}{2}\int_{\kappa \cup \kappa '} \div \u_h^{i+1}C_h^{i+1}S_e\,d\x\\
& +\int_{\kappa \cup \kappa '}( \u_h^{i+1}\cdot \nabla (C_h^{i+1}-C))S_e\,d\x +\int_{\kappa \cup \kappa '} ((\u_h^{i+1} -\u)\cdot \nabla C)S_e\,d\x.
\end{aligned}
\end{equation}
In view of the continuity properties of $\mathcal{L}_e$ in Property \ref{psii}, a bound for the above left-hand side is derived by the same arguments; for instance, by combining it with \eqref{eq:inversin}, we have on the elements $\kappa$ sharing $e$:
\begin{equation*}
    \left\|\mathcal{L}_e(v)\right\|_{L^6(\kappa)} \leq c c_I(6)h_{\kappa}^{-d/3}h_e^{1/2}\left\|v\right\|_{L^2(e)}.
\end{equation*}
Thus, by applying \eqref{etaid11}, we obtain
\begin{equation} \label{etaid12}
h_e^{\frac{1}{2}}\left\|\alpha[\nabla C_h^{i+1}\cdot \n]_e \right\|_{L^2(e)} \leq c \big( \left\|\u-\u_h^{i+1}\right\|_{L^2(\kappa \cup \kappa ')} + |C-C_h^{i+1}|_{H^1(\kappa \cup \kappa ')} + h_e \left\|g-g_h\right\|_{L^2(\kappa \cup \kappa ')} \big)
\end{equation}
and thus, we get the desired result.

$\hfill\Box$

 \end{proof}
Now, we turn to the velocity error indicators.
\begin{thm} Let $(\u,p,C)$ and $(\u_h^{i+1},p_h^{i+1},C_h^{i+1})$ be the respective solutions to problems $(V_a)$ and $(V_{ahi})$. We have the following bounds of the indicators: for each element $\kappa \in \Gamma_h$,
\begin{equation} \label{etail1}
    \eta_{\kappa,i}^{(L_1)} \leq \left\|\u-u_h^i\right\|_{L^2(\kappa)}+\left\|\u-\u_h^{i+1}\right\|_{L^2(\kappa)}
\end{equation}
and
\begin{equation} \label{etaid3}
\eta_{\kappa,i}^{(D_3)} \leq c\left\|\v_r\right\|_{L^3(w_k)}
\end{equation}
where $c$ is a positive constant independent of $h$.
\end{thm}
\begin{proof}
The bound \eqref{etail1} is a simple consequence of the definition of $\eta_{\kappa,i}^{(L_1)}$ and a triangle inequality.
In order to prove \eqref{etaid3}, we consider first Equation \eqref{residupression} with $q_h=0$ and
\[
q=q_\kappa = \left\{
\begin{array}{lcl}
(\div \u_h^{i+1}) \psi_\kappa & \mbox{on } \kappa, \\
0 & \mbox{on } \Omega \backslash \kappa,
\end{array}
\right.
\]
where $\psi_\kappa$ is the bubble functions on a given element $\kappa \in \mathcal{T}_h$. We obtain by using Relation \eqref{perpond} the following equation:
\begin{equation} \int_{\kappa} (\div \u_h^{i+1})^2 \psi_\kappa \,d\x = \int_{\kappa} \nabla q_k \cdot \v_r\,d\x.
\end{equation}
Then we use property \ref{psi}, the Cauchy-Schwartz inequality and the relation $\left\|\v\right\|_{L^2(\kappa)} \leq |\kappa|^{1/6}\left\|\v\right\|_{L^3(\kappa)} \leq h_{\kappa}^{d/6}\left\|\v\right\|_{L^3(\kappa)} $:
\begin{equation} \begin{aligned}
 \left\|\div \u_h^{i+1}\right\|_{L^2(\kappa)} &\leq ch_{\kappa}^{-1}\left\|\v_r\right\|_{L^2(\kappa)} \\
& \leq ch_{\kappa}^{(-6+d)/6}\left\|\v_r\right\|_{L^3(\kappa)}.
\end{aligned}
\end{equation}
Then we get by using the first inverse inequality \eqref{eq:inversin} with $p=3$ and by multiplying by $h_{\kappa}$
\begin{equation} \label{hkdiv}
h_{\kappa} \left\|\div \u_h^{i+1}\right\|_{L^3(\kappa)} \leq c_2 \left\|\v_r\right\|_{L^3(\kappa)}
\end{equation}
which corresponds to the first divergence part of the indicator $\eta_{\kappa , i}^{(D_3)}$.\\
Again,  we consider Equation \eqref{residupression} with $q_h=0$ and
\begin{equation*}
q=q_{e}=
\left \{
\begin{array}{lcl}
\mathcal{L}_{e,\kappa } \big(
 \phi_{h,1}^e \psi_{e}  \big)& \hspace{-0.5cm} \mbox{on } \{ \kappa, \kappa' \}, \\
0 & \hspace{-0.3cm}\mbox{ on } \Omega \backslash  (\kappa \cup \kappa') , \\
\end{array}
 \right.
\end{equation*}
\noindent where $\psi_e$ is the bubble function of $e$ and $\kappa'$ denotes the other element of $\mathcal{T}_h$ that share  $e$ with $\kappa$. We get the following equation:
\begin{equation*}
\ds \int _{e} (\phi_{h,1}^e)^2 \psi_e\,d\s
= \ds \int_{\kappa \cup \kappa'} \div \u_h^{i+1}q_e\,d\x - \int_{\kappa \cup \kappa'} \nabla q_e\cdot \v_r  \,d\x
\end{equation*}
Properties \ref{psi} and \ref{psii} allow us to get the following bound:
\begin{equation*}
    \left\|\phi_{h,1}^e\right\|_{L^2(e)} \leq c\big(h_e^{1/2}\left\|\div \u_h^{i+1}\right\|_{L^2(\kappa \cup \kappa ')} + h_e^{-1/2}\left\|\v_r\right\|_{L^2(\kappa \cup \kappa ')} \big).
\end{equation*}
By using the second inverse inequality \ref{eq:inversin}, the relation $\| \v_h \|_{L^2(\kappa)} \le |\kappa|^{1/6} \| \v_h \|_{L^3(\kappa)}$ and that the family of triangulation is uniformly regular, we obtain the bound:
\begin{equation} \label{etaid32}
    h_e^{1/3}\left\|\phi_{h,1}^e\right\|_{L^3(e)} \leq c\big(h_e \left\|\div \u_h^{i+1}\right\|_{L^3(\kappa \cup \kappa ')}+\left\|\v_r\right\|_{{L^3(\kappa \cup \kappa ')}}\big).
\end{equation}
Hence, we bound the part of $\eta_{\kappa,i}^{(D_3)}$ corresponding to $\phi_{h,1}^e$. Relations \eqref{hkdiv} and \eqref{etaid32} give Relation \eqref{etaid3}.

$\hfill \Box$

\end{proof}
\begin{thm} \label{etaid2}
Under the assumptions of Lemma \ref{uborne6}, we have the following bound
\begin{equation} \begin{aligned}
\eta_{\kappa,i}^{(D_2)} & \leq \hat{c} \big(  \eta_{\kappa,i}^{(L_1)} + \| \u - \u_h^{i+1} \|_{L^2(w_\kappa)} + \|  \frac{\beta}{\rho}
\big( |\u| \u - |\u_h^{i+1}| \u_h^{i+1}  \big) + \nabla ( p-p_h^{i+1} )\|_{L^2(\kappa)} \\
&+ \left\|\f(.,C)-\f_h(.C)\right\|_{L^2(w_{\kappa})} + \| K^{-1} - K_h^{-1} \|_{L^{3}(w_\kappa)} + c_{\f_1}\left\|C-C_h^{i+1}\right\|_{H^1(w_{\kappa})}
\end{aligned}
\end{equation}
where $\hat{c}$ is a constant independent of the mesh step but depends on the exact solution $(\u,p)$ and $K_h^{-1}$ is an approximation of $K^{-1}$ which is a constant tensor in each triangle.
\end{thm}
\begin{proof}
Let us now prove Relation \eqref{etaid2}. We consider Equation \eqref{residuvitesse} with $\v_h=0$ and
\begin{equation*} \v=\v_{\kappa}=
\left \{
\begin{array}{lcl}
\ds \big( -\nabla p_h^{i+1}- \gamma
(\u_h^{i+1} - \u_h^{i}) -
\frac{\mu}{\rho}
K_h^{-1} \u^{i+1}_h - \frac{\beta}{\rho} |\u^i_h| \u^{i+1}_h + \f_h(.,C_h^i) \big)\psi_{\kappa} && \hspace{-0.cm} \mbox{ on } \kappa, \\
0 && \hspace{-0.cm}\mbox{ on } \Omega \backslash \kappa, \\
\end{array}
 \right.
\end{equation*}
where $K_h^{-1}$ is an approximation of $K^{-1}$ which is a constant tensor in each triangle.
We obtain the following equation:
\begin{equation*}
\begin{array}{ll}
\medskip
 \ds  \int_\kappa | (-\nabla p_h^{i+1}- \gamma
 (\u_h^{i+1} - \u_h^{i}) -
 \frac{\mu}{\rho}
 K_h^{-1} \u^{i+1}_h - \frac{\beta}{\rho} |\u^i_h| \u^{i+1}_h + \f_h(.,C_h^i)) \psi_\kappa^{1/2} |^2 d\x = \\
 \medskip
 \hspace{2cm} \ds \frac{\mu}{\rho}
 \int_\kappa   (K^{-1} - K_h^{-1}) \u^{i+1}_h \cdot \v \, d\x +
 \frac{\mu}{\rho} \int_\kappa K^{-1} (\u - \u^{i+1}_h) \cdot \v \, d\x +
 \frac{\beta}{\rho} \int_\kappa (|\u| \u - |\u^i_h| \u^{i+1}_h) \cdot \v \, d\x \\
 \medskip
 \hspace{2cm}  \ds + \ds\int_{\kappa} \nabla  ( p-p_h^{i+1} ) \cdot \textbf{v}\,d\x   - \gamma \ds \int _{K} (\u_h^{i+1} - \u_h^i) \cdot \v \, d\x, + \int_{\kappa} (\f(.,C)-\f_h(.,C))\cdot \v \,d\x \\
 \medskip
 \hspace{2cm} \ds + \ds \int_{\kappa} (\f_h(,C)-\f_h(.,C_h^i))\cdot \v \,d\x,
\end{array}
\end{equation*}
and then by using Lemma \ref{uborne6}, the bound \eqref{fhlipshitz} and Properties \ref{psi} and \ref{psii}, we get  the following bound:
\begin{equation} \begin{aligned}
& \left\|-\nabla p_h^{i+1}-
 \gamma (\u_h^{i+1} - \u_h^{i}) -  \frac{\mu}{\rho}
 K^{-1}_h \u^{i+1}_h - \frac{\beta}{\rho} |\u^i_h| \u^{i+1}_h + \f_h(.,C_h^i)\right\|_{L^2(\kappa)} \\
 & \leq \hat{c} \big( \left\|K^{-1}-K_h^{-1}\right\|_{L^3(\kappa)}+\frac{\mu K_M}{\rho} \left\|\u -\u_h^{i+1}\right\|_{L^2(\kappa)} + \left\| \frac{\beta}{\rho} ( |\u|\u - |\u_h^i|\u_h^{i+1}) + \nabla ( p-p_h^{i+1}]\right\|_{L^2(\kappa)} \\
 &+ \gamma \left\|\u_h^{i+1}-\u_h^i\right\|_{L^2(\kappa)} + \left\|\f(.,C)-\f_h(.,C)\right\|_{L^2(\kappa)} + c_{\f_1} \left\|C-C_h^{i+1}\right\|_{H^1(\kappa)} \big).
\end{aligned}
\end{equation}
Finally, we get the result by using the following triangle inequality and Lemma \ref{uborne6}:
\begin{equation}\label{aux11}
\begin{array}{ll}
\medskip
 \ds  \| -\nabla p_h^{i+1}-
 \gamma (\u_h^{i+1} - \u_h^{i}) -  \frac{\mu}{\rho}
 K^{-1} \u^{i+1}_h - \frac{\beta}{\rho} |\u^i_h| \u^{i+1}_h + \f_h(.,C_h^i)
 \|_{L^2(\kappa)} \le \\
 \medskip
 \hspace{.5cm} \ds   \| -\nabla p_h^{i+1}-
 \gamma (\u_h^{i+1} - \u_h^{i}) -  \frac{\mu}{\rho}
 K^{-1}_h \u^{i+1}_h - \frac{\beta}{\rho} |\u^i_h| \u^{i+1}_h + \f_h(.,C_h^i)
 \|_{L^2(\kappa)} + \frac{\mu}{\rho} \| K^{-1}  - K^{-1}_h \|_{L^3(\kappa)} \|  \u^{i+1}_h \|_{L^6(\kappa)}.
\end{array}
\end{equation}
\end{proof}
\section{Numerical results} \label{numerique}
The main goal of this section is to validate the theoretical results of the previous sections, all numerical simulations are in two dimensions and performed using Freefem++ (see \cite{hecht}). In this part, we will show two cases : the first one is an academic one when the numerical solution is compared to the known exact one, the second case treats the Led-Driven cavity which a very popular and interesting one.
\subsection{First test case}
In this subsection, the domain $\Omega$ is the unit square $]0,1[ \times ]0,1[$ and all computations start on a uniform initial triangular mesh obtained by dividing the domain into $N^2$ equal squares, each one subdivided into two triangles, so that the initial triangulation consists of $2N^2$ triangles.  We apply the numerical scheme $(V_{ahi})$ to the exact solution $(\u,p,C)=(\curl \psi,p,C)$ where $\psi, p,$  and $C$ are given by
\begin{equation}
\label{eq:1.78}
\ds \psi(x,y)=e^{-  \delta  ((x-0.5)^2+(y-0.5)^2)},
\end{equation}
\begin{equation}
    \label{pressionexacte}
    p(x,y)=x*(x-2/3)*y*(y-2/3),
\end{equation}
and
\begin{equation}
    \label{concentrexacte}
    C(x,y)=x^2*(x-1)^2*y^2*(y-1)^2*e^{-  \delta  ((x-0.5)^2+(y-0.5)^2)}
\end{equation}
with the choice $\delta=50; \gamma=\beta=10;  \alpha=\mu=\rho=r_0=1$; $K=I$ and $\f_1(C)=(2+C,2+2sin(C)).$ Thus, we compute $\f_0$ and $g$ by using their expressions in problem $(P)$.
For the choice of the parameter $\gamma$, we refer to \cite{semaan2} where we compared the numerical scheme for different values of the parameter $\gamma$. In addition, we take $N=20$ on the initial mesh.\\

The theory is tested by applying the numerical scheme $(V_{ahi})$ to the exact solution, we consider the total error:
$$Err=\ds \Big( \frac{||\u_h^{i+1}-\u||_{L^3(\Omega)} + ||\nabla(p_h^{i+1}-p)||_{L^{3/2}(\Omega)}+||C_h^{i+1}-C||_{H^1(\Omega)}}{||\u||_{L^3(\Omega)} + ||\nabla p||_{L^{3/2}(\Omega)} + ||C||_{H^1(\Omega)}} \Big).$$
It should be noted that in the definition of $Err$ we considered the norm of $||\u_h^{i+1}-\u ||$ in $L^3(\Omega)$, where the velocity lives, although in the definition of the indicator $\eta_{\kappa,i}^{(L_1)}$ we considered it in $L^2(\Omega)^d$.\\
For the computation of the numerical solution by using the scheme $(V_{ahi})$, it is convenient to compute the following global indicators:
\begin{equation}\nonumber
 \eta_{i}^{(D)}=\Big(\sum _{K \in
 \mathcal{T}_h } \big((\eta_{K,i}^{(D_1)})^2 + (\eta_{K,i}^{(D_2)})^2+(\eta_{K,i}^{(D_3)})^2\big)^{\frac{1}{2}}
\Big)
\end{equation}
and
\begin{equation}\nonumber
 \eta_{i}^{(L)}=\Big(\sum _{K \in
 \mathcal{T}_h }\big((\eta_{K,i}^{(L_1)})^2+(\eta_{K,i}^{(L_2)})^2\big)^{\frac{1}{2}}
\Big)
\end{equation}
where the indicators $\eta_{K,i}^{(D_i)} $ and $\eta_{K,j}^{(L_j)} $ with $i\in \{1,2,3\}$ and $j \in \{1,2\} $ are given in equations \eqref{premindic}-\eqref{dernindic}.
These indicators are used
for the stopping criteria given by the relation
\begin{equation} \label{stoppingg}
    \eta_{i}^{(L)} \leq \bar{\gamma} \eta_{i}^{(D)}
\end{equation}
where $\bar{\gamma}=0.01$. This criteria was first introduced in \cite{LAM11} and
\cite{ERN11} and used in \cite{semaan2}.

For the adaptive mesh (refinement and coarsening), we use routines in Freefem++. The indicators are used for mesh adaptation by the adapted mesh algorithm used in \cite{semaan2}, but here we add the convection-diffusion-reaction equation and for reader's convenience, we prefered to recall the algorithm:
\begin{enumerate}
\item[(1)] Given $(\u^i_h,p_h^i,C_h^i)$,
\begin{enumerate}
    \item[(a)]  Solve the problem $(V_{ahi})$ to compute $(\u_{h}^{i+1},p_h^{i+1},C_h^{i+1})$.
    \item[(b)] Calculate $\eta_i^{(D)}$ and $\eta_i^{(L)}$.
\end{enumerate}
\item[(2)] If the stopping criterion \eqref{stoppingg} is satisfied, go to (3), else set $\u_h^i=\u_h^{i+1}, C_h^i=C_h^{i+1}, p_h^i=p_h^{i+1}$ and go to (1).
\item[(3)] $\mbox{ }$
\begin{enumerate}
\item[(a)] If $\eta_i^{(D)}$ is smaller than a fixed error tolerance $\varepsilon=10^{-8}$, we stop the iterations and the algorithm.
\item[(b)] Else we adapt the mesh using the indicators $\eta_{K,i}^{(D)}$.
\end{enumerate}
\item[(4)] Set $i=i+1$ and go to (1).
\end{enumerate}
In Figure \ref{figure1}, we present the evolution of the mesh
during the iterations (initial, second and fourth refinement levels). We notice
that the mesh is concentrated in the region where the solution needs to be well described as the velocity is a ball concentrated at the center of $\Omega$.\\
\begin{figure}[ht!]
\hskip-1.8cm
\begin{subfigure}[b]{0.35\textwidth}
\centering
\includegraphics[width=7.5cm]{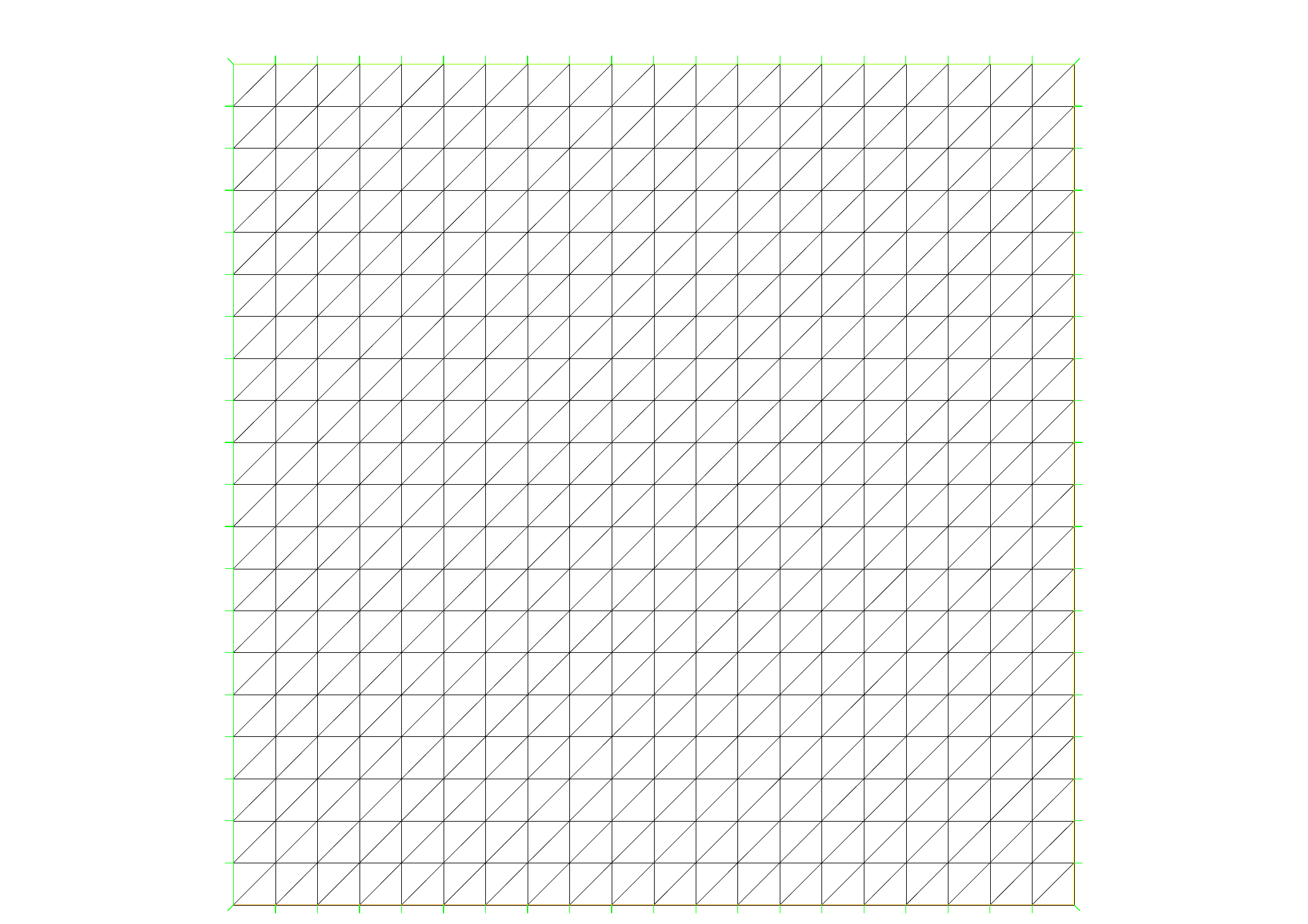}
\end{subfigure}
\hskip-.5cm
\begin{subfigure}[b]{0.35\textwidth}
\centering
\includegraphics[width=7.5cm]{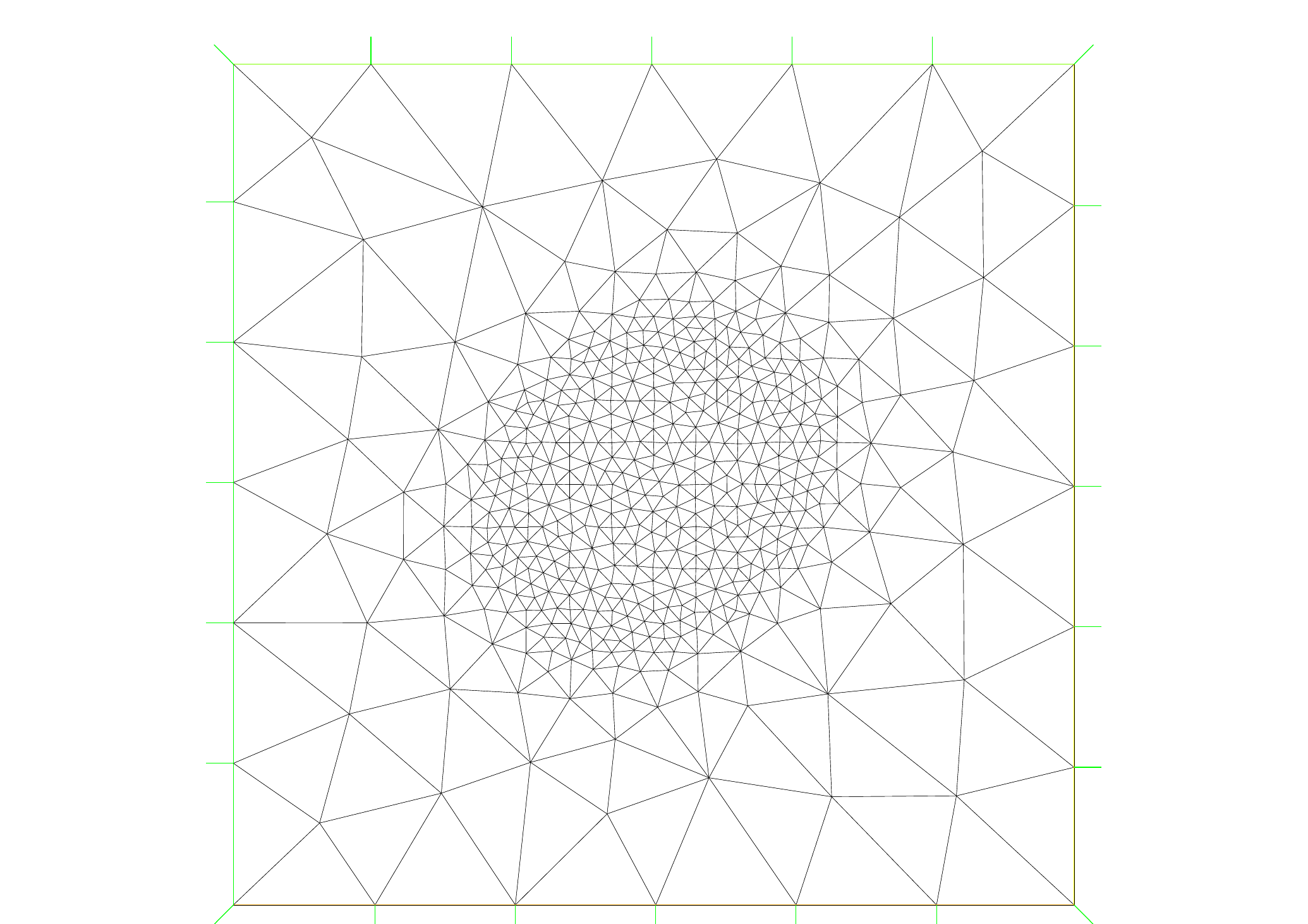}
\end{subfigure}
%
\hskip-.5cm
\begin{subfigure}[b]{0.35\textwidth}
\centering
\includegraphics[width=7.5cm]{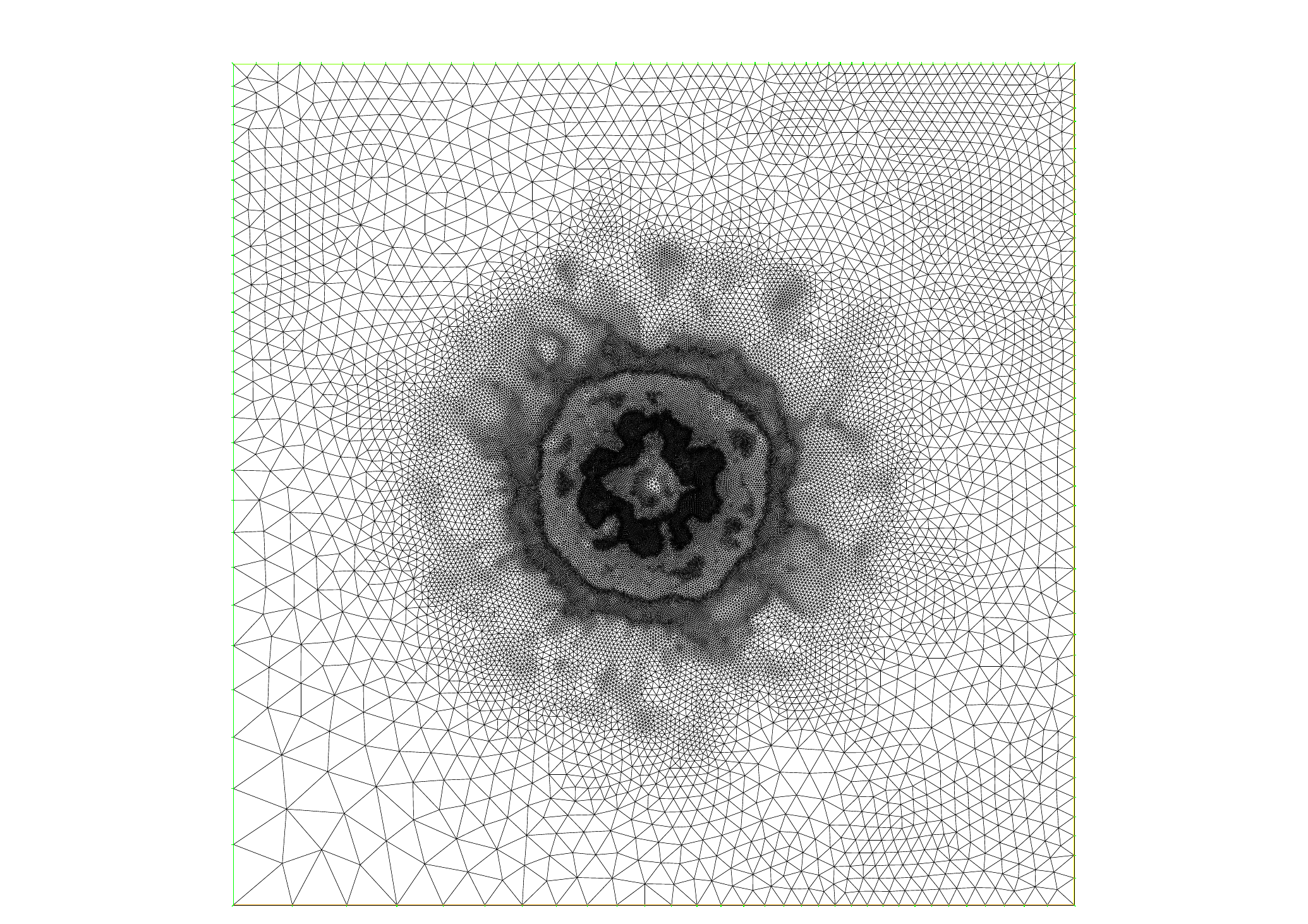}
\end{subfigure}
\caption{Evolution of the mesh during the refinement levels (initial, second and fifth).}\label{figure1}
\end{figure}

Tables \eqref{tableerruni} and \eqref{tableerradap} show the rate of convergence of the error $Err$ in logarithmic scale for the uniform and adaptive method.
\begin{table}[h!]
    \begin{tabular}{|l|l|l|}
  \hline
\bf  total degree of freedom &   \bf Err &  \bf rate \\
  \hline
  $3.87489$ & $-1.00064$ & $-$ \\
  \hline
  $4.28375$ & $-1.46323$ & $-1.13141418$ \\
  \hline
  $4.56799$ & $-1.76405$ & $-1.05833099$ \\
  \hline
  $4.79535$ & $-1.99443$ & $-1.0132829$ \\
  \hline
  $4.98563$ & $-2.1783$ & $-0.9663128$ \\
  \hline
  $5.13388$ & $-2.2859$ & $-0.72580101$\\
  \hline
  $5.26413$ & $-2.43679$ & $-1.15846449$\\
  \hline
  $5.38048$ & $-2.509$ & $-0.621315$\\
  \hline
  $5.48576$ & $-2.63355 $ & $-1.18303571$\\
  \hline
    \end{tabular}
    \caption{  Uniform method: rate of the error $E_{rr}$ with respect to the  total degree of freedom in
logarithmic scale.}
\label{tableerruni}
\end{table}
\begin{table}[h!]
    \begin{tabular}{|l|l|l|}
  \hline
  \bf  total degree of freedom & \bf  Err & \bf  rate \\
  \hline
  $3.87489$ & $-1.00064$ &    - \\
  \hline
  $4.09156$ & $-1.71735$ & $-3.30784142$ \\
  \hline
  $4.35629$ & $-1.97805$ & $-0.98477694$ \\
  \hline
  $4.64669$ & $-2.55506$ & $-1.98694904$ \\
  \hline
  $5.07834$ & $-2.97641$ & $-0.97613807$ \\
  \hline
  $5.47644$ & $-3.3199$ & $-0.86282341$ \\
  \hline
    \end{tabular}
    \caption{ Adaptive  method: rate of the error $E_{rr}$ with respect to the  total degree of freedom in
logarithmic scale.}
\label{tableerradap}
\end{table}
To go far with our numerical studies, we plot and study the error curves between the exact and numerical solutions corresponding to our problem.
Figure \eqref{courbes} plots the comparison of the global error curves versus the total degree of freedom in logarithmic scale. We notice that the errors of the adaptive mesh method are smaller than those given by the uniform method. Hence the efficiency of the adaptive mesh method.\\
\begin{figure}[h!]
\centering
\includegraphics[width=9.5cm]{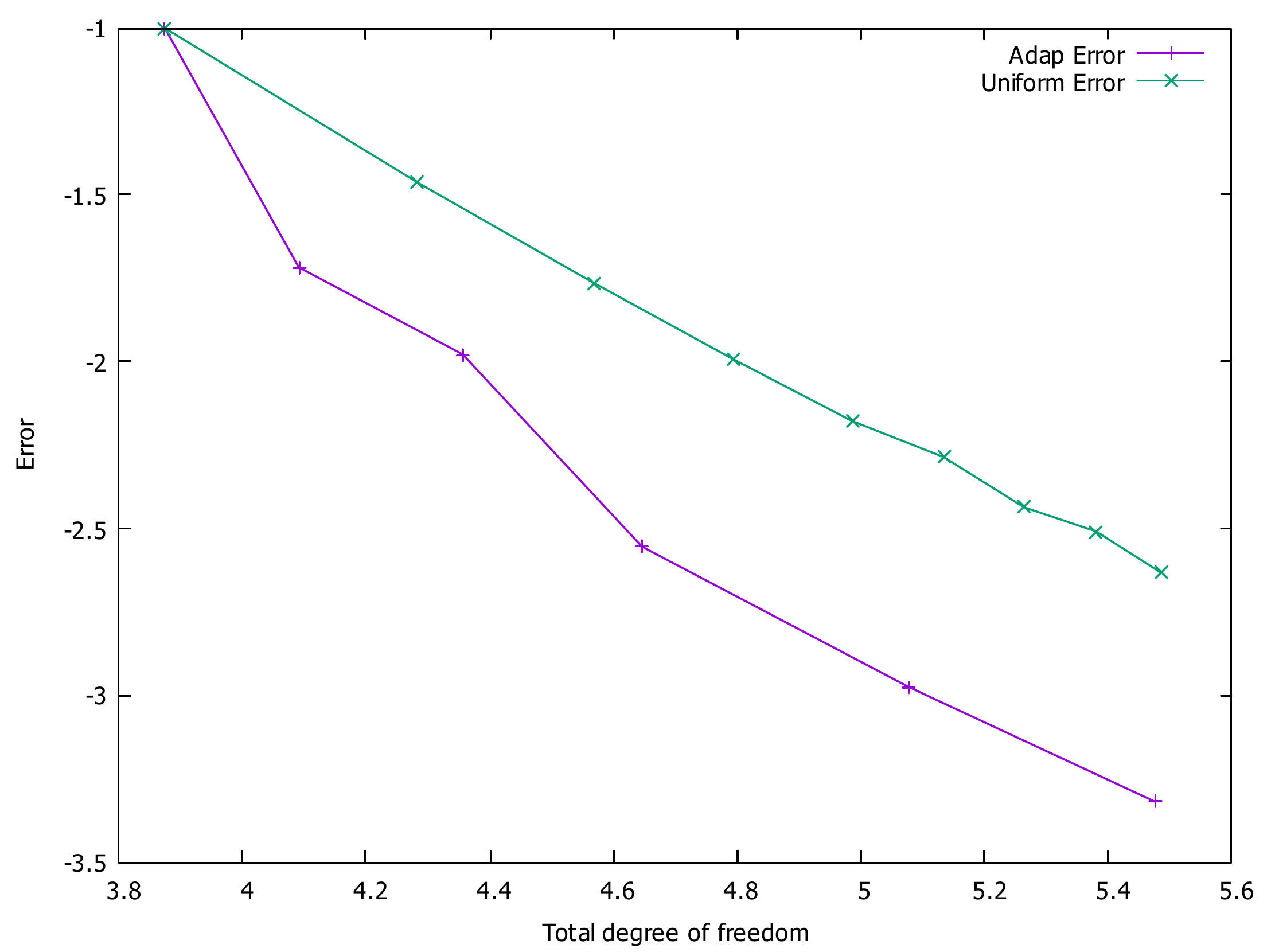}
\vskip -.4cm
\caption{Comparison of the errors $Err$ with respect to the total degree of freedom in logarithmic scale.}\label{courbes}
\end{figure}
In table \eqref{tableff}, we present the effectivity index defined as :
\[
EI = \ds \frac{\ds \eta_i^{(L)} +  \eta_i^{(D)}}{\ds
 \|\textbf{u}-\textbf{u}_h^{i+1}\|_{L^3(\Omega)} + ||\nabla(p - p_h^{i+1|}) ||_{L^{3/2}(\Omega)}+ ||C-C_h^{i+1}||_{H^1(\Omega)} }
\]
with respect to the number of vertices during the refinement levels. We remark that it decreases from $33.3389$ (first refinement level) to $9.98932$ (refinement level $6$).
\begin{table}[h!]
\begin{tabular} {|l||*{18}{c|}}
 \hline Refinement Level &  initial  & first  & second  & third  & fourth & fifth  \\
 \hline
 Number of vertices
 &441   & 485    &1481    & 4323  &15088  & 44940   \\
 \hline Effectivity index  & 33.3389 & 24.4027 & 12.2602 & 13.1712 & 13.1547 & 9.98932 \\
\hline
\end{tabular}
\vspace{.2cm} \caption{ EI with respect to the refinement levels.}
\label{tableff}
\end{table}
\subsection{Second test case (Driven cavity):}
The driven cavity is a test of performance algorithms in fluid problems. It was used in several works and among them we cite \cite{semaan,Schr83,jaddakroub}. 
In this subsection, we show numerical simulations corresponding to this test in order to study the {\it{a posteriori}} error estimates and the efficiency of the proposed method.
We suppose that $\Omega=]0,1[^2$, $K=I$, $\mu=r_0=1$, $\beta=20$, $\gamma=10$, $\f_0=\0$, $\f_1(C)=(10C,10C)$ and $g=0$. We complete the Darcy-Forchheimer equation with the boundary condition $\u.\n=0$ in $\partial \Omega$, and the convection-diffusion-reaction equation with the boundary condition $C=20 x(x-1)y(y-1)$ on the top $\Gamma_1$ of $\Omega$ and $C=0$ on $\partial \Omega \backslash \Gamma_1$.
Again, we consider an uniform initial mesh with  $N=20$ and we begin by showing comparisons between the uniform and adaptive methods corresponding to problem $(V_{ahi})$.
Figure \eqref{figure2}
present the evolution of the mesh during the iterations. We can see that, from an iteration to another, the concentration of the refinement is on the complex vorticity regions and at the top boundary corresponding to $y=1$.
\begin{figure}[ht!]
\hskip-1.8cm
\begin{subfigure}[b]{0.35\textwidth}
\centering
\includegraphics[width=7.5cm]{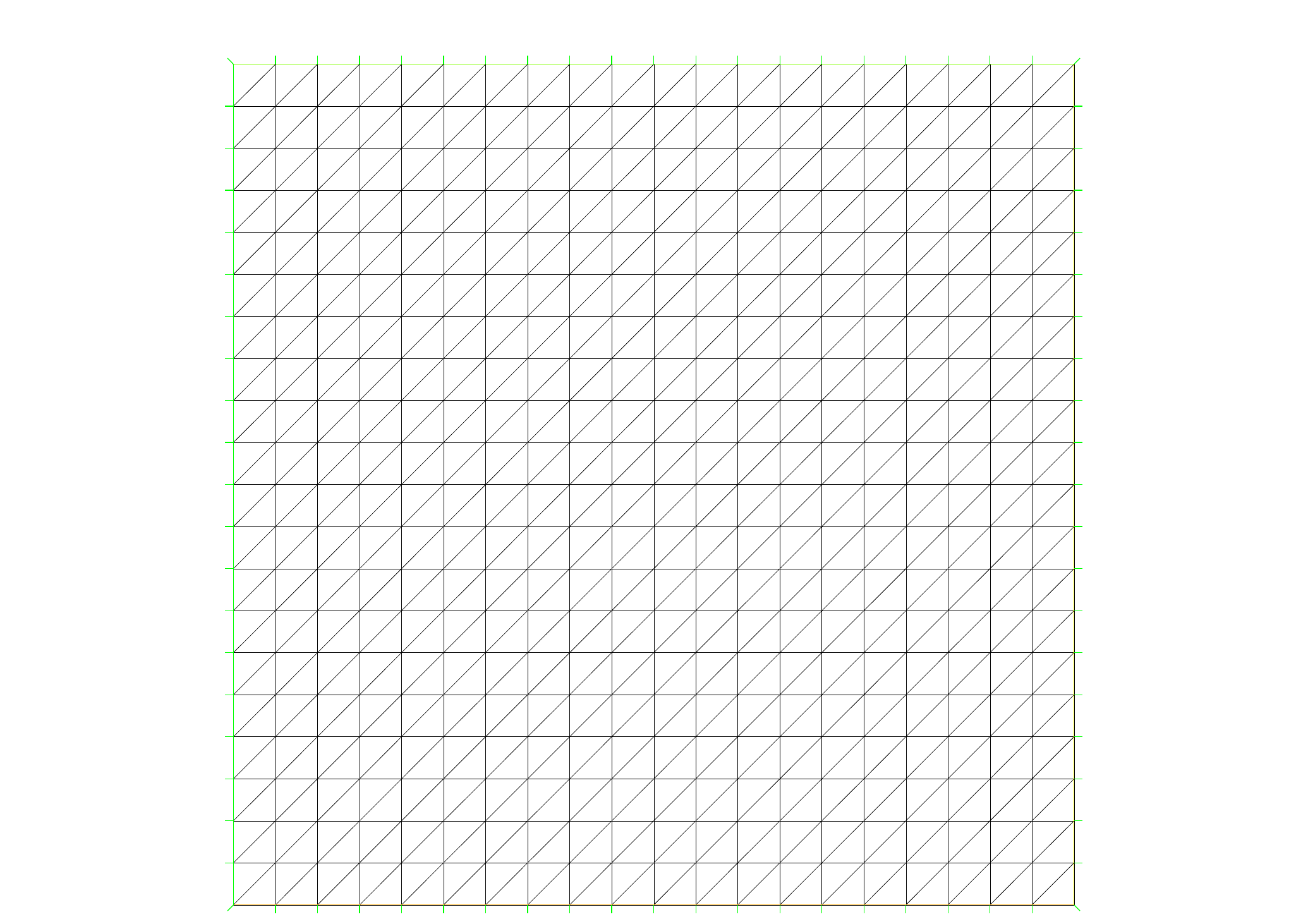}
\end{subfigure}
\hskip-.5cm
\begin{subfigure}[b]{0.35\textwidth}
\centering
\includegraphics[width=7.5cm]{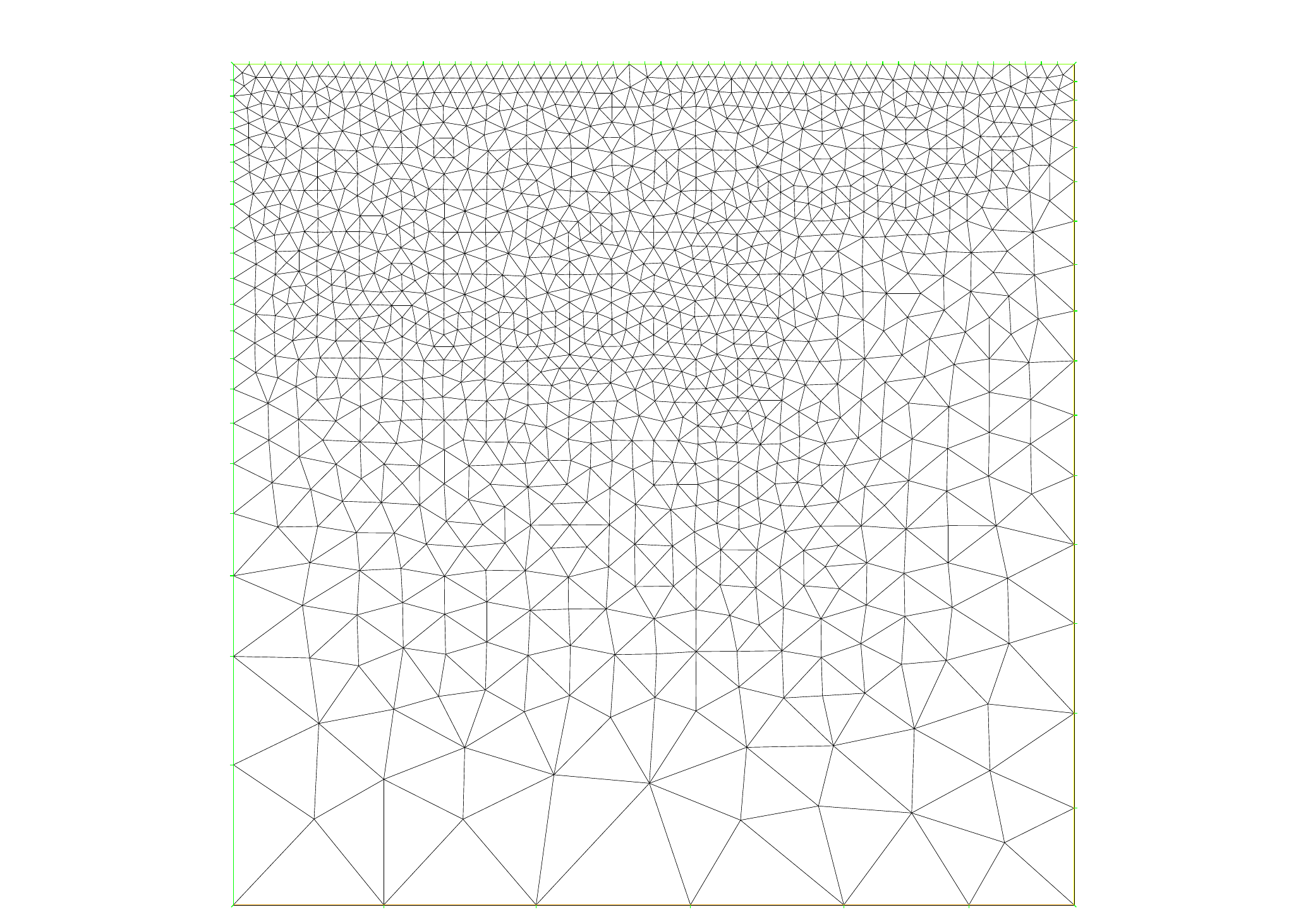}
\end{subfigure}
%
\hskip-.5cm
\begin{subfigure}[b]{0.35\textwidth}
\centering
\includegraphics[width=7.5cm]{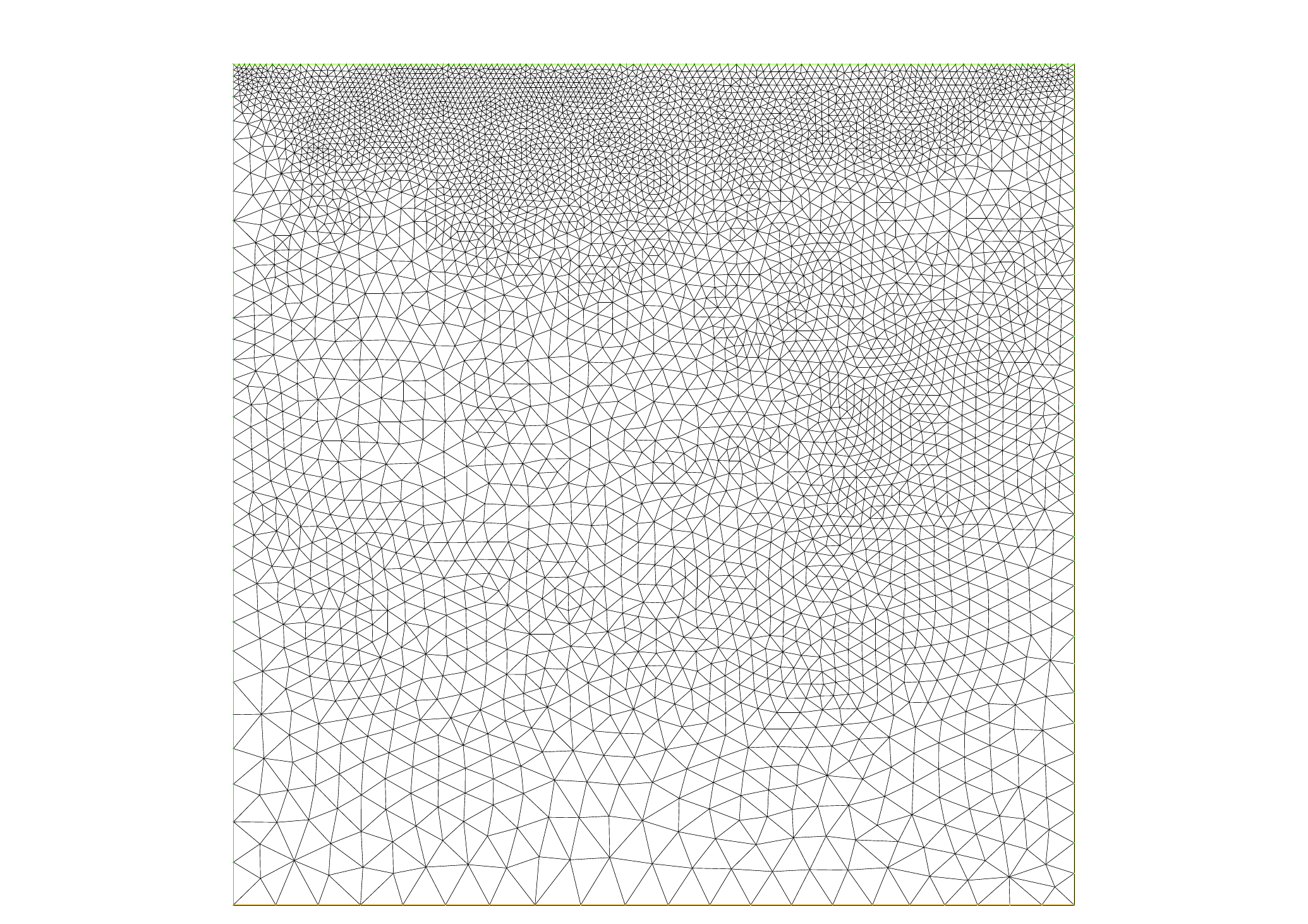}
\end{subfigure}
\caption{Evolution of the mesh during the refinement levels (initial, second and third).}\label{figure2}
\end{figure}
In Figures \eqref{velocity2_5}-\eqref{concentration2_5}, we consider the color velocity and concentration at the second refinement level. We remark that the solution is more important where the refinement of the mesh is concentrated (see figure \eqref{figure2}).
\begin{figure}[htbp]
\begin{minipage}[b]{0.49\linewidth}
\hspace{-.5cm}
\includegraphics[width=10cm]{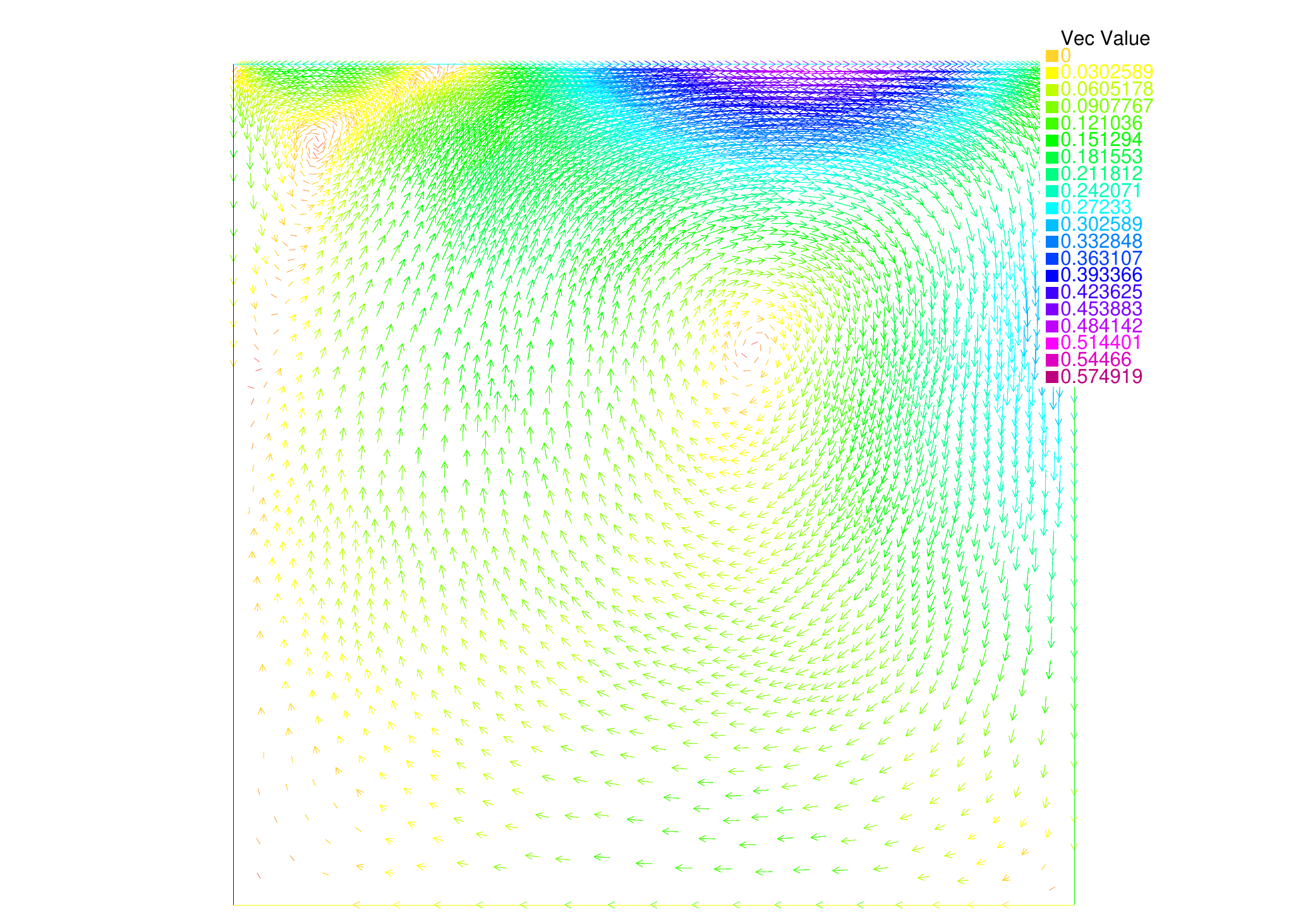}
\vspace{-.4cm}
\caption{numerical velocity at the fourth refinement level.} \label{velocity2_5}
\end{minipage}\hfill
\begin{minipage}[b]{0.49\linewidth}
\hspace{-.5cm}
\includegraphics[width=10cm]{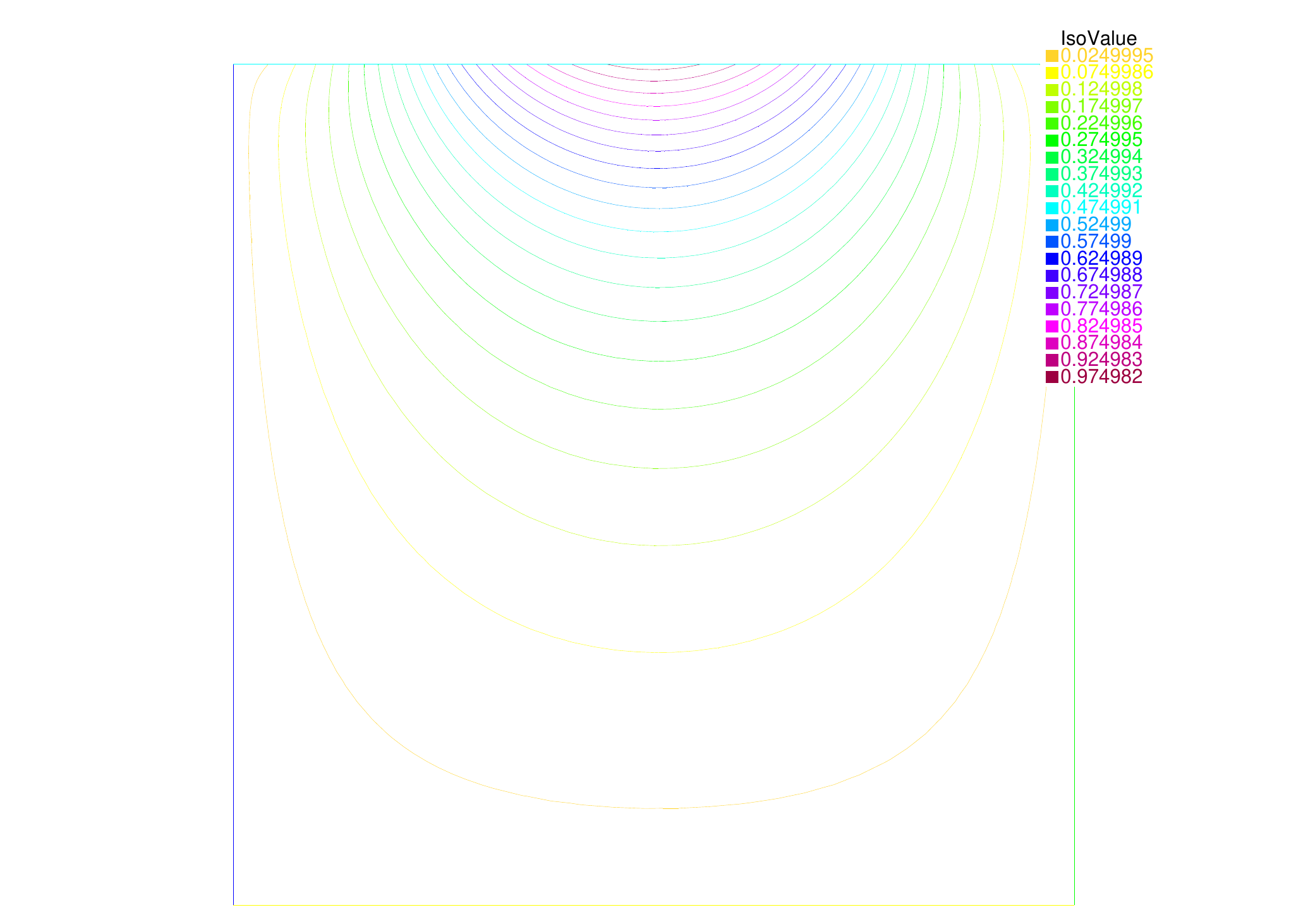}
\vspace{-.4cm}
\caption{numerical concentration at the fourth refinement level.} \label{concentration2_5}
\end{minipage} \hfill
\end{figure}

Now, we introduce the relative total error to the indicator given by
\[
E_{tot} = \ds \frac{\ds   \eta_i^{(D)}}{\ds
 \|\textbf{u}_h^i\|_{L^3(\Omega)} + ||\nabla(p_h^i) ||_{L^{3/2}(\Omega)}+ ||C_h^i||_{H^1(\Omega)} }
\]
where here also $\eta_i^{(D)}$ is computed after convergence on the iterations $i$ (by using the stopping criteria \eqref{stoppingg}).
In figure \eqref{test2courbe}, we plot the relative total error $E_{tot}$ for the uniform and the adaptive methods. Note that $E_{tot}$ represents the global indicator errors (while $E_{rr}$ in the previous case represents the total error between the exact and numerical solutions). We remark, like the first test case, that for the same total degree of freedom, the adaptive error is much smaller than the uniform error.
\begin{figure}[h!]
\centering
\includegraphics[width=9.5cm]{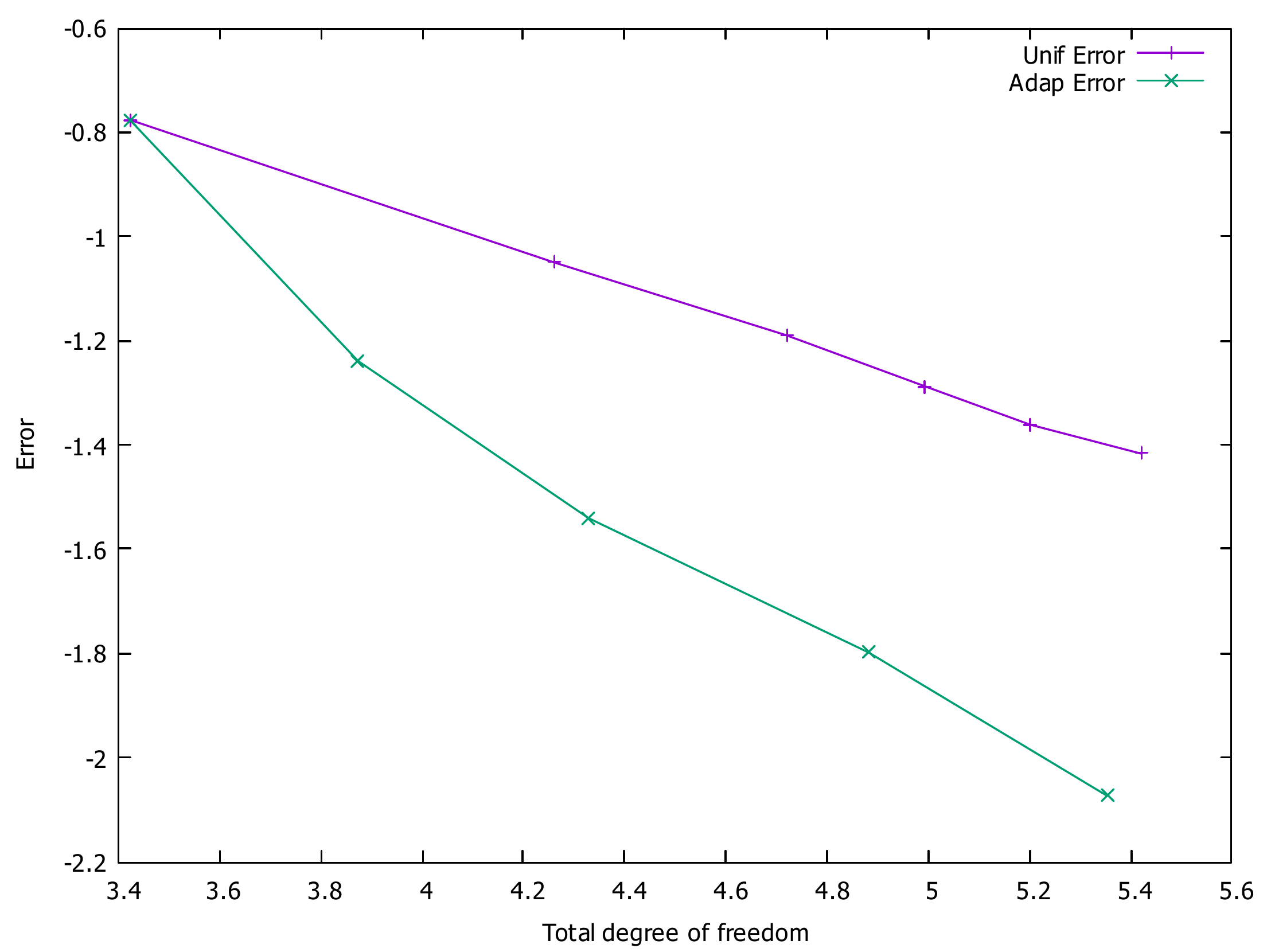}
\vskip -.4cm
\caption{Comparison of the errors $E_{tot}$ with respect to the total degree of freedom in logarithmic scale.}\label{test2courbe}
\end{figure}

\textbf{Conclusion}:
In this work, we introduced the variational formulation of the  Darcy-Forchheimer problem coupled with the convection-diffusion-reaction equation. We discretized the problem by using finite element method. We then  constructed  error indicators to evaluate
the errors of the numerical approximation. Finally, we performed several numerical simulations where the indicators are used for mesh adaptation, confirming the efficiency of the adaptive methods.

\end{document}